\documentclass[11pt, reqno]{amsart}
\usepackage{amscd,amsfonts,amsmath,amssymb,amstext,amsthm}
\usepackage[all]{xy}
\usepackage{xcolor, enumerate, verbatim}
\usepackage{hyperref}
\hypersetup{
    colorlinks=true,
    linkcolor=black,
    filecolor=cyan,          
    pdfpagemode=FullScreen,
    }
\usepackage{enumitem}
\makeatletter
\theoremstyle{plain}
\newtheorem{theorem} {Theorem} [section]
\newtheorem{lemma} [theorem]{Lemma}
\newtheorem{proposition}[theorem]{Proposition}
\newtheorem{corollary} [theorem]{Corollary}

\theoremstyle{definition}

\newtheorem{remark}[theorem]{Remark}
\numberwithin{equation}{section}

\newcommand{\g}{\mathfrak{g}} 
\newcommand{\Aut}
{\mathrm{Aut}}

\newcommand{\aut}{\mathfrak{aut}}
\newcommand{\aff}
{\mathfrak{aff}}

\newcommand{\ip}[1]{\left\langle#1\right\rangle}
\newcommand{\RE}{\mathrm{Re}\,}

\title{A K\"ahler potential on the unit ball \\with constant differential norm}
\date\today
\author{Kang-hyurk Lee and Aeryeong Seo}
\address{Department of Mathematics and Research Institute of Natural Science, Gyeongsang
National University, Jinju, Gyeongnam, 52828, The Republic of Korea}
\email{nyawoo@gnu.ac.kr}
\address{Department of Mathematics and RIRCM, Kyungpook National University,
80, Daehak-ro, Buk-gu, Daegu, 41566, Republic of Korea}%
\email{aeryeong.seo@knu.ac.kr}


\begin{document}
\maketitle
\tableofcontents
\markboth{Kang-Hyurk Lee, Aeryeong Seo}{A K\"ahler potential with constant differential norm}

\begin{abstract}
Let $\mathbb B^n$ be the unit ball in $\mathbb C^n$ and $\mathbb H^n$ be the homogeneous Siegel domain of the second kind which is biholomorphic to $\mathbb B^n$.
We show that the K\"ahler potential of $\mathbb H^n$ is unique up to the automorphisms among K\"ahler potentials whose differentials have constant norms.

As an application, we consider a domain $\Omega$ in $\mathbb C^n$, which is biholomorphic to $\mathbb B^n$. 
We show that if $\Omega$ is affine homogeneous, then it is affine equivalent to $\mathbb H^n$. Assume next that its canonical potential with respect to the K\"ahler--Einstein metric has a differential with a constant norm. 
If the biholomorphism between $\Omega$ and $\mathbb B^n$ is a restriction of a M\"obius transformation, then the map is affine equivalent to a Cayley transform.
\end{abstract}

\section{Introduction}
For a K\"ahler manifold $M$ with its K\"ahler form $\omega$, we say that a function $\eta\colon M\to \mathbb R$ is a (K\"ahler) \emph{potential} if it satisfies 
$\sqrt{-1}\partial\overline \partial \eta = \omega$. 
If $\|\partial\eta\|_\omega$ is constant on $M$ where $\|\partial\eta\|_\omega$ is the norm of $\partial\eta$ measured by $\omega$, then we will say that $\eta$ has a \emph{constant differential norm} with respect to $\omega$. If the metric that we are using to measure the norm is obvious, then we will just say that $\eta$ has a constant differential norm.

Since Gromov's pioneering work on the $L^2$ cohomology theory related to the norm $\|d\eta\|_{\omega}$ \cite{Gromov1991}, a significant amount of research has been conducted in this area.
At first, Gromov introduced the idea of a \emph{K\"ahler hyperbolic} manifold; a complex manifold $M$ is said to be K\"ahler hyperbolic if its universal cover admits a K\"ahler metric $g$ and there exists its potential $\phi$ such that $\|\phi\|_g<\infty$. He presented in the same paper that a complete simply connected K\"ahler hyperbolic manifold $M$, which admits a compact quotient, has $\mathcal H^{p,q}=0$ for $p+q\neq \dim_{\mathbb C}M$ and $\mathcal H^{p,q}\neq 0$ for $p+q=\dim_{\mathbb C}M$, where $\mathcal H^{p,q}$ denotes the space of harmonic $L^2$ forms on $M$ of bidegree $(p,q)$. For the references on the study of K\"ahler hyperbolic manifolds, see \cite{BDET2022} and the references therein.

\medskip

One of the interesting results was proved by Kai--Ohsawa \cite{Kai-Ohsawa2007}, where they showed that any bounded homogeneous domain $B$ has a potential $\phi$ of the Bergman metric with constant differential norm. 
In \cite{CLY2020}, Choi--Lee--Yoo showed that on the unit disc $\Delta$ such potential of the Poincar\'e metric is unique up to the automorphisms of $\Delta$ assuming that the constant is unique for potentials with constant norms. Based on this property, they characterized the upper half space as a domain admitting a canonical potential whose differential norm is a certain constant.
Recently Choi--Lee--Seo \cite{CLS23} showed that for a simply connected complex manifold $M$ having dimension $n$ that covers a compact complex manifold and admits a complete K\"ahler--Einstein metric $\omega$ with negative Ricci curvature $-K$, if 
$M$ admits a potential $\varphi\colon M\to\mathbb R
$ of $\omega$ satisfying
$\|\partial\varphi\|_\omega^2\equiv  \frac{n+1}{K}$, then $M$ is biholomorphic to the unit ball.
See also \cite{LKH2021}.

\medskip

In this paper we generalize the result of Choi--Lee--Yoo \cite{CLY2020} to the higher dimensional ball $\mathbb B^n$ with $n\geq 2$. 
Let 
$$
\mathbb B^n := \{ z\in \mathbb C^n : |z|<1\}
$$ 
be the unit ball in $\mathbb C^n$
and 
$$
\mathbb H^n := \{w\in \mathbb C^n : \textup{Re } w_n > |w_1|^2 +\dots + |w_{n-1}|^2\}
$$
be the Siegel domain of the second kind 
which is biholomorphic to $\mathbb B^n$ by the Cayley transform $\mathcal{C} \colon \mathbb B^n\rightarrow \mathbb H^n$ given by
$$
\mathcal{C}(z) = \left( \frac{z_1}{1-z_n}, \ldots, \frac{z_{n-1}}{1-z_n}, \frac{1+z_n}{1-z_n}\right).
$$
Let $\omega_{\mathbb H^n}$ and $\omega_{\mathbb B^n}$ be the invariant K\"ahler--Einstein metrics of the Ricci curvature $-1$ on $\mathbb H^n$ and $\mathbb B^n$, respectively. 
The canonical (K\"ahler) potential of $\omega_{\mathbb H^n}$ is given by $\log\psi_0$ where
$$\psi_0(w):=(\text{Re}w_n - |w_1|^2-\cdots -|w_n|^2)^{-n-1}$$
and the pullback of $\log \psi_0$ by $\mathcal C$, which is a K\"ahler potential of $\omega_{\mathbb B^n}$, is given by $\log\varphi_0$ where 
\begin{equation}\nonumber
    \varphi_0(z):=\psi_0\circ\mathcal C(z)= \frac{|1-z_n|^{2(n+1)}}{\left(1-|z|^2\right)^{n+1}}.
\end{equation}
The squares of their  differential norms, $\|\partial\log \psi_0\|_{\omega_{\mathbb H^n}}^2$ and $\|\partial\log\varphi_0\|^2_{\omega_{\mathbb B^n}}$, are constant $n+1$. 
Interestingly, if there exists a potential on $\mathbb B^n$ for $\omega_{\mathbb B^n}$ that has constant differential norm, then the norm constant should be $n+1$ (see Theorem~\ref{1-dim thm} for $\Delta$ and Theorem~\ref{uniqueness of constant} for $\mathbb B^n$ with $n\geq 2$).

\smallskip

Our first main theorem is as follows:
\begin{theorem}\label{main1}
Let $\omega_{\mathbb H^n}$ be the Bergman--Poincar\'e metric on $\mathbb H^n$. 
Suppose that there exists a positive real valued function $\psi\colon \mathbb H^n\to \mathbb R$ such that $\log \psi$ is a K\"ahler potential of $\omega_{\mathbb H^n}$ and $\|\partial\log \psi\|_{\omega_{\mathbb H^n}}$ is constant on $\mathbb H^n$. Then $\log\psi$ is the canonical potential of $\mathbb H^n$ up to isotropy subgroup of $\mathbb H^n$ at $(0,\ldots,0,1)\in \mathbb H^n$.
\end{theorem}
Theorem~\ref{main1} can be rephrased on $\mathbb B^n$ by the Cayley transformation in the subsequent way:
\begin{corollary}\label{thm_ball}
Any K\"ahler potential of the Bergman--Poincar\'e metric on the unit ball, which has constant differential norm, is $\log\varphi_0$ up to the isotropy subgroup of $\mathbb B^n$ at $0$.
\end{corollary}

For $n=1$, Theorem~\ref{main1} and Corollary~\ref{thm_ball} are proved under certain conditions 
in \cite[Theorem 1.2] {CLY2020}. In the same paper, as an application, the authors showed that if $\Omega$ is a simply connected, proper domain in $\mathbb C$ and $\omega_\Omega:= \sqrt{-1}\lambda dz\wedge d\overline z$ is a complete Hermitian metric with the Gaussian curvature $-1$, then $\Omega$ is affine equivalent to $\mathbb H^1$ if and only if $\|\partial \log\lambda\|_{\omega_{\Omega}}$ is a constant. By the Riemann mapping theorem, note that such $\Omega$ is biholomorphic to the unit disc $\Delta$.

Now let $\Omega$ be a bounded domain in $\mathbb C^n$ such that $\Omega$ is biholomorphic to the unit ball $\mathbb B^n$ and equip the complete K\"ahler--Einstein metric $\omega_\Omega= \sqrt{-1}\sum_{i,j=1}^ng_{i\bar j}dz_i\wedge d\overline z_j$ with Ricci curvature $-1$. Since $\omega_\Omega$ has constant Ricci curvature, the natural K\"ahler potential is $\log\det(g_{i\bar j})$.
However, unlike the case $n=1$, the condition $\|\partial\log\det g_{i\bar j}\|_{\omega_\Omega}\equiv c$ does not imply that $\Omega$ is affine equivalent to $\mathbb H^n$.
For example, consider a holomorphic map $\widetilde{\mathcal C}\colon\mathbb B^n\to \mathbb C^n$ given by 
\begin{equation}\nonumber
    \widetilde{\mathcal C}(z) = \left( \frac{z_1}{1-z_n} + g_1(z_n), \cdots, \frac{z_{n-1}}{1-z_n}+g_{n-1}(z_n), \frac{1+z_n}{1-z_n}\right)
\end{equation}
for some holomorphic functions $g_j\colon \Delta\to\mathbb C$, $j=1,\ldots, n-1$.
Let $\Omega:=\widetilde{\mathcal C}(\mathbb B^n)$ be the image of $\widetilde{\mathcal C}$. Then by a simple calculation we obtain $\det \partial\widetilde{\mathcal C} = (1-z_n)^{-(n+1)}$ and $\widetilde{\mathcal C}$ is a biholomorphism onto $\Omega$.
Then by \eqref{isometry} we have $$
\det g_{i\bar j}\circ \widetilde{\mathcal C} = \frac{|1-z_n|^{2(n+1)}}{(1-|z|^2)^{n+1}},
$$
and it implies that $\|\partial \log\det g_{i\overline j}\|_{\omega_\Omega}$ is a constant because $\widetilde{\mathcal C}$ is an isometry.
However we can choose $g_j$, $j=1,\ldots, n-1$, such that $\Omega$ is not affine equivalent to $\mathbb H^n$.
Note that it is possible to have convex $\Omega$.
Based on this observation we need some strong conditions to characterize the domain biholomorphic to the unit ball in the sense of affine equivalence. 

\smallskip
Our second main theorem is as follows:
\begin{theorem}\label{main2}
Any affine homogeneous domain in $\mathbb C^n$ biholomorphic to the unit ball is affine equivalent to $\mathbb H^n$.
\end{theorem}
To prove Theorem~\ref{main2}, we exploit the Lie algebra structure of homogeneous Siegel domains of the second kind, which was thoroughly studied by Kaup--Matsushima--Ochiai in \cite{KMO1970}. In particular, we characterize the complete holomorphic vector fields, which consist of the basis of $\aff(\mathbb H^n)$.

\medskip

In \cite{Mok-Tsai1992}, Mok--Tsai showed that if a bounded symmetric domain $B\subset \mathbb C^n$ with rank
$r\geq 2$ is biholomorphic to an unbounded convex domain, then the map should be a Cayley transform up to automorphisms of $B$ and affine  transformations of $\mathbb C^n$. 
To prove it, they showed that the biholomorphism can be extended as an automorphism of the compact dual of $B$. Next, they figured out what this automorphism is. 
In the case of $\mathbb B^n$, the bounded symmetric domain with rank $1$, we can readily see that the first part of the proof does not hold. 
On the other hand, we realized that it is possible to characterize the automorphisms of $\mathbb C\mathbb P^n$, which maps $\mathbb B^n$ onto the domain of which the K\"ahler potential has constant differential norm.

Let $\mathbb C\mathbb P^n$ be the complex projective space of dimension $n$ and let $[z]=[z_1,\ldots, z_{n+1}]$ be its homogeneous coordinates. Then, any nondegenerate $(n+1)\times (n+1)$ matrix $A=(A_{ij})$ can be considered as a holomorphic automorphism of $\mathbb C\mathbb P^n$ given by \begin{equation}\label{linear map}
[z]\mapsto [Az] = \left[\sum_j A_{1j}z_j, \cdots,
\sum_j A_{n+1 j}z_j\right].
\end{equation}
Consider a canonical embedding of $\mathbb C^n$ into $\mathbb C\mathbb P^n$ given by $(z_1,\ldots, z_n)\hookrightarrow [z_1,\ldots, z_n, 1]$. Then the realization of the map \eqref{linear map} on $\mathbb C^n$ is given by
\begin{equation}\label{Mobius transform}
z\in \mathbb C^n \mapsto \left[\frac{\sum_j A_{1j}z_j}{\sum_j A_{n+1 j}z_j}, \cdots,
\frac{\sum_j A_{nj}z_j}{\sum_j A_{n+1 j}z_j}\right].
\end{equation}
This map of the form \eqref{Mobius transform} is called a {\it M\"obius transformation}.

Our last theorem is as follows:
\begin{theorem}\label{main3}
Let $\Omega$ be a domain in $\mathbb C^n$ and $G\colon \mathbb B^n\to \Omega$ be a biholomorphism which is a restriction of a M\"obius transformation.
Let $\omega_\Omega = \sum_{i,j=1}^n g_{i\bar j} dz_i\wedge d\overline z_j$ be the K\"ahler--Einstein metric on $\Omega$ with the Ricci curvature $-1$.
Suppose that 
$
\| \partial\log\det (g_{i\bar j})\|_{\omega_\Omega}
$ 
is constant. Then $\Omega$ and $G$ are affine equivalent to $\mathbb H^n$ and $\mathcal C$, respectively.
\end{theorem}

This paper is organized as follows: 
Section~\ref{1-dim} presents the proof for the uniqueness of the K\"ahler potential of the Poincar\`e metric on $\mathbb H^1$ up to automorphisms.
Section~\ref{preliminaries} describes the geometry of $\mathbb H^n$, for example,  K\"ahler--Einstein metrics of $\mathbb H^n$, K\"ahler potentials and parabolic/hyperbolic vector fields will be presented. 
Moreover,
we will prove that the K\"ahler potential with constant differential norm of the Bergman metric on $\mathbb H^n$, $n\geq 2$ is unique up to automorphisms.
Subsection~\ref{hyperbolic} and \ref{parabolic} present some properties of the K\"ahler potential that vanishes by the real part of the hyperbolic and parabolic vector fields, respectively. 
Section~\ref{proof of main1} proves Theorem~\ref{main1}.
Section~\ref{Lie algebra structure of Hn} characterizes the complete holomorphic vector fields that correspond to the Lie algebra decomposition of $\aut(\mathbb H^n)$ given by Kaup--Matsushima--Ochiai in \cite{KMO1970}. 
Section~\ref{proof of main23} proves Theorem~\ref{main2} and Theorem~\ref{main3}.

\bigskip

{\bf Acknowledgement} This work was supported by Samsung Science and Technology Foundation under Project Number SSTF-BA2201-01. The  second  named author was supported by the National Research Foundation of Korea (NRF) grant funded by the Korea government (No. NRF-2022R1F1A1063038)

\section{Uniqueness of the constant when $n=1$}\label{1-dim}
Let $\mathbb H:=\mathbb H^1:=\{w\in \mathbb C : \text{Re}\,w>0\}$ be the right half plane and $\omega_{\mathbb H}:= -2\sqrt{-1}\partial\bar\partial\log \text{Re}\, w$ be the K\"ahler--Einstein metric on $\mathbb H$.
Let $\log\psi_0(w)=\log \frac{1}{(\text{Re}\, w)^2}$ be the canonical K\"ahler potential of $\omega_{\mathbb H}$.
Note that $\|\partial \log \psi_0\|_{\omega_{\mathbb H}}^2 = 2$.
In this section we will denote $\frac{\partial f}{\partial w}$ by $f'$ for simplicity for any holomorphic function $f\colon \mathbb H\to \mathbb C$.
For a vector field $V$ we denote $\frac{V+\overline V}{2}$ by $\text{Re}\,V$.

\smallskip 
The aim of this section is to prove 
\begin{theorem}\label{1-dim thm}
    Let $\psi\colon \mathbb H\to\mathbb R$ be a function such that $\log\psi$ is a potential of the Bergman metric on $\mathbb H$. Suppose that 
    $$
    \|\partial \log\psi\|^2_{\omega_{\mathbb H}}\equiv c
    $$
    for some constant $c\in \mathbb R$.
    Then $c=2$ and $\psi=r\psi_0$ for some $r>0$.
\end{theorem}

Consider the holomorphic vector fields of $\mathbb H$:
\begin{equation}\nonumber    
T := 2\sqrt{-1}\frac{\partial}{\partial w} \quad\text{ and } \quad
D:= 2w\frac{\partial}{\partial w},
\end{equation}
 which generate affine automorphisms:
 $$
 \mathcal T_s(w) := w+2\sqrt{-1}s
 \quad\text{ and } \quad
 \mathcal D_s(w) := e^{2s}w ,
 $$
 respectively, with $s\in \mathbb R$.
By a straightforward calculation, we derive
\begin{equation}\nonumber
(\text{Re}\, D) \log \psi_0=-2 \quad\text{ and }\quad (\text{Re}\, T)\log\psi_0 = 0
\end{equation}
and 
\begin{equation}\label{1-dim pushforward}
    (\mathcal T_s)_*D = D-2sT\quad\text{ and }\quad
    (\mathcal T_s)_* T = T.
\end{equation}
\begin{lemma}\label{1-dim hyperbolic}
    Let $\psi\colon\mathbb H\to\mathbb R$ be a K\"ahler potential of $\omega_{\mathbb H}$ with 
    $$
    \|\partial \log\psi\|^2_{\omega_{\mathbb H}}\equiv c_1 
    \quad\text{and } \quad
    (\textup{Re}\, D)\log \psi \equiv c_2
    $$
    for some constant $c_1$ and $c_2$. Then $c_1=2$ and $c_2=\pm 2$.
\end{lemma}
\begin{proof}
    Since $\log\psi$ and $\log\psi_0$ are the potential of $\omega_{\mathbb H}$, there exists a holomorphic function $f\colon \mathbb H\to \mathbb C$ such that 
    $\log\psi -\log\psi_0 = f + \bar f$.
    Since we have 
    \begin{equation}\nonumber 
        \begin{aligned}
            c_2 &= (\text{Re}\, D)\log\psi 
            =(\text{Re}\, D)(\log\psi_0 + f+\bar f)
            = -2 + \text{Re}\, (Df),
        \end{aligned}
    \end{equation}
    we have 
    \begin{equation}\label{1-dim f'} 
    wf'(w) = \frac{c_2+2}{2} + \sqrt{-1}\beta
    \end{equation}
    for some constant $\beta\in\mathbb R$. 
    Denote $C:= \frac{c_2+2}{2} + \sqrt{-1}\beta$. 
    Note that $c_2$ is a real number.

    On the other hand, since we have 
    \begin{equation}\label{1-dim_constant norm eq}
    \begin{aligned}
    c_1 &= \|\partial \log\psi\|^2_{\omega_{\mathbb H}}
    = \|\partial\psi_0 + \partial f\|^2_{\omega_{\mathbb H}}
    = \left| -\frac{1}{\text{Re}\, w} + f'\right|^2 2(\text{Re}\, w)^2\\
    &= 2 + 2(\text{Re}\, w)^2|f'|^2 - 4\text{Re}\, w\,\text{Re}\, f',
    \end{aligned}
    \end{equation}
    by \eqref{1-dim f'} we obtain
    \begin{equation}\label{1-dim eqn}
        (c_1-2)|w|^2 = 2 (\text{Re}\, w)^2 |C|^2 - 4\text{Re}\, w\,\text{Re}(C\bar w).
    \end{equation}
    By differentiating \eqref{1-dim eqn} with respect to $\frac{\partial^2}{\partial w^2}$ we obtain
    \begin{equation}\label{second deriv1}
        0= |C|^2 -2\overline C
    \end{equation}
    and by differentiating it again with respect to $\frac{\partial^2}{\partial w\partial \bar w}$ we obtain
    \begin{equation}\label{second deriv2}
        c_1-2 = |C|^2 - 2\text{Re}\,\overline C.
    \end{equation}
    By \eqref{second deriv1}, $C$ is real and hence by \eqref{second deriv1} and $\eqref{second deriv2}$ we obtain $c_1=2$ and $C=0$ or $2$. This implies that $c_2=\pm 2$.
\end{proof}

\begin{lemma}\label{1-dim parabolic}
    Let $\psi\colon\mathbb H\to\mathbb R$ be a K\"ahler potential of $\omega_{\mathbb H}$ with 
    $$
    \|\partial \log\psi\|^2_{\omega_{\mathbb H}}\equiv c_1 
    \quad\text{ and } \quad
    (\textup{Re}\, T)\log \psi \equiv c_2
    $$
    for some constant $c_1$ and $c_2$. Then $c_1=2$, $c_2=0$ and $\psi=r\psi_0$ for some $r>0$.
\end{lemma}
\begin{proof}
Since $\log\psi$ and $\log\psi_0$ are the potential of $\omega_{\mathbb H}$, there exists a holomorphic function $f\colon \mathbb H\to \mathbb C$ such that 
    $\log\psi -\log\psi_0 = f + \bar f$.
    Since we have 
    \begin{equation}\nonumber 
        \begin{aligned}
            c_2 &= (\text{Re}\, T)\log\psi 
            =(\text{Re}\, T)(\log\psi_0 + f+\bar f)
            = \text{Re}\, (Tf),
        \end{aligned}
    \end{equation}
    we obtain 
    \begin{equation}\label{1-dim f'_2}
        f'(w) = \beta -\sqrt{-1}\frac{c_2}{2}=:C.
    \end{equation}
     On the other hand by \eqref{1-dim_constant norm eq} we have 
    \begin{equation}\nonumber
        c_1-2 = 2 (\text{Re}\, w)^2 |C|^2 - 4\text{Re}\, w\,\text{Re}\, C.
    \end{equation}
    By comparing the coefficient of monomials in $w$ and $\overline w$, we obtain $c_1=2$ and $C=0$ which implies $c_2=0$. 
    Since $f$ is a constant by \eqref{1-dim f'_2}, the proof is complete.
\end{proof}

\medskip
{\bf Proof of Theorem~\ref{1-dim thm}:}
    Under the assumptions of the theorem, there exists a nowhere vanishing holomorphic complete vector field $W$ on $\mathbb H$ given in \cite[Theorem~3.2]{Choi-Lee2021} satisfying $(\text{Re}\,W)\log \psi\equiv0$.
    By taking some isotropy automorphism of $\mathbb H$, we can assume that $W$ vanishes at infinity. Since the dimension of the automorphism group of $\mathbb H$ is $3$ and that of the Lie subgroup fixing infinity is $2$, 
    we can express $W$ as
    $$
    W = aD + bT
    $$
    for some constants $a,\,b\in \mathbb R$. 

    {\bf Step 1:} Suppose that $a\neq 0$. Using \eqref{1-dim pushforward} we obtain
    \begin{equation}
        (\mathcal T_s)_* W 
        = aD+(b-2sa)T.
    \end{equation}
    If we take $s=\frac{b}{2a}$, then we have $\widetilde W:=(\mathcal T_s)_* W = aD$.
    Furthermore, $\log\widetilde\psi$ is also a K\"ahler potential of $\omega_{\mathbb H}$ with $\widetilde \psi:= \psi\circ \mathcal T_{-s}\colon\mathbb H\to \mathbb R$ and $\|\partial\log\widetilde\psi\|^2_{\omega_{\mathbb H}}\equiv c$ since $\mathcal T_s$ is a holomorphic isometry with respect to $\omega_{\mathbb H}$.
    By Lemma~\ref{1-dim hyperbolic}, we obtain $(\text{Re}\,\widetilde W)\log\psi = a(\text{Re}\,D)\log\psi=\pm 2a$
    and this contradiction implies $a=0$.

    {\bf Step 2:} Since $W$ is nowhere vanishing, we have $W=bT$ with $b\neq 0$.
    This implies that $0=(\text{Re}\, W)\log\psi = b(\text{Re}\, T)\log\psi$ and hence by Lemma~\ref{1-dim parabolic}
    we have $c=2$ and $\psi=r\psi_0$. \qed
\begin{corollary}\label{1-dim cor}
    Let $\psi\colon \mathbb H\to\mathbb R$ be a function such that $\log\psi$ is a potential of the complete K\"ahler--Einstein metric $\widetilde \omega$ on $\mathbb H$ with Ricci curvature $-\kappa$ for $\kappa>0$. Suppose that 
    $$
    \|\partial \log\psi\|^2_{\widetilde\omega}\equiv c
    $$
    for some constant $c\in \mathbb R$.
    Then $c=\frac{2}{\kappa}$ and $\psi=r\psi_0^{1/\kappa}$ for some $r>0$.
\end{corollary}
\begin{proof}
    Since the Bergman metric on $\mathbb H^n$ is the K\"ahler--Einstein metric with Ricci curvature $-1$, we have $\widetilde \omega=\frac{1}{\kappa}\omega_{\mathbb H^n}$. This implies that $\kappa \log\psi$ is a K\"ahler potential of the Bergman metric satisfying the condition in Theorem~\ref{1-dim thm}. 
    Therefore, $c\kappa = 2$ and $\psi^\kappa = r\psi_0$ for some $r>0$. This completes the proof.
\end{proof}

\section{Geometry of $\mathbb H^n$}\label{preliminaries}
In this section, we present some properties of K\"ahler potentials of the Bergman metric of $\mathbb H^n$. We will also show that if the Bergman metric of $\mathbb H^n$ admits a potential with constant differential norm $c$, then $c=n+1$ (Section~\ref{unique constant}). Moreover, we derive an equation induced from constant differential norm condition and we describe the behavior of the holomorphic vector field on $\mathbb H^n$ which generates affine automorphisms.
The information presented in this section will be used crucially to prove Theorem~\ref{main1}.
\subsection{Canonical K\"ahler potentials of the Bergman metrics.}
Let $\mathbb B^n := \{ z\in \mathbb C^n : |z|<1\}$ be the unit ball in $\mathbb C^n$
and 
$$
\mathbb H^n := \{w\in \mathbb C^n : \textup{Re } w_n > |w_1|^2 +\dots + |w_{n-1}|^2\}
$$
be the Siegel domain of the second kind 
which is biholomorphic to $\mathbb B^n$ by the Cayley transform $\mathcal{C} \colon \mathbb B^n\to\mathbb H^n$ given by
$$
\mathcal{C}(z) = \left( \frac{z_1}{1-z_n}, \ldots, \frac{z_{n-1}}{1-z_n}, \frac{1+z_n}{1-z_n}\right).
$$
The inverse $\mathcal F:= \mathcal C^{-1}\colon \mathbb H^n\to\mathbb B^n$ of $\mathcal C$ is given by $$\mathcal{F}(w) = \left( \frac{2w_1}{ w_n + 1 },\,\ldots, \,  \frac{2w_{n-1}}{ w_n + 1 },\, \frac{w_n-1}{ w_n + 1 }\right).
$$
Let $\omega_{\mathbb H^n}$ be the K\"ahler--Einstein metric with the Ricci curvature $-1$ on $\mathbb H^n$, i.e. for $w':=(w_1,\ldots, w_{n-1})$
\begin{equation}\nonumber
\begin{aligned}
    \omega_{\mathbb H^n} &= 
    -(n+1)\sqrt{-1}\partial\bar\partial \log(\text{Re}\,w_n - |w'|^2).
    \end{aligned}
\end{equation}
Let $ (g_{i\bar j})$ denote the matrix form of $\omega_{\mathbb H^n}$ with respect to the standard Euclidean coordinates.
We remark that $\omega_{\mathbb H^n}$ is the Bergman metric of $\mathbb H^n$.
Let $(g^{k\overline j})$ denote the inverse of $(g_{i\overline j})$ satisfying $\sum_{j}g_{i\bar j}g^{k\bar j}=\delta_{ik}$. Explicitly,  
\begin{equation}\nonumber
    (g^{k\overline j}) = \frac{\text{Re}\,w_n - |w'|^2}{n+1}
    \left(\begin{array}{ccc|c}
    1& &0&2 w_1\\
    &\ddots&&\vdots\\
    0&&1&2 w_{n-1}\\\hline
    2\overline w_1&\cdots&2\overline w_{n-1}& 4\text{Re}\, w_n
    \end{array}\right).
\end{equation}
The canonical potential of the $\omega_{\mathbb H^n}$ is  $\log\psi_0$ where
$$
\psi_0(w):=(\textup{Re } w_n - |w'|^2)^{-n-1} 
$$
for $w=(w_1,\ldots, w_n)=(w',w_n)$.
Remark that 
\begin{equation}\nonumber
\begin{aligned}
\|\partial\log\psi_0\|_{\omega_{\mathbb H^n}}^2
&= \frac{(n+1)^2}{(\text{Re}\,w_n-|w'|^2)^2}\|\partial(\text{Re}\,w_n-|w'|^2)\|^2_{\omega_{\mathbb H^n}}\\
&=\frac{n+1}{(\text{Re}\,w_n-|w'|^2)}
X   \left(\begin{array}{ccc|c}
    1& &0&2 w_1\\
    &\ddots&&\vdots\\
    0&&1&2 w_{n-1}\\\hline
    2\overline w_1&\cdots&2\overline w_{n-1}& 4\text{Re}\, w_n
    \end{array}\right)
    \overline X^t\\
    &= \frac{n+1}{(\text{Re}\,w_n-|w'|^2)}\left(0,\cdots,0,2(\text{Re}\,w_n - |w'|^2)\right)\overline X^t\\
    &=n+1
\end{aligned}
\end{equation}
where $X=\left(-w_1,\cdots, -w_{n-1},1/2\right)$ which represents $\partial(\textup{Re}\,w_n - |w'|^2)$ in Euclidean coordinates.

The pullback function of $\psi_0$ to $\mathbb B^n$ by $\mathcal C$ is given by 
\begin{equation}\nonumber
    \varphi_0(z):=\psi_0\circ\mathcal C(z)= \frac{|1-z_n|^{2(n+1)}}{\left(1-|z|^2\right)^{n+1}}
\end{equation}
for $z=(z_1,\ldots, z_n)$.
Let $\omega_{\mathbb B^n}$ be the K\"ahler--Einstein (Bergman) metric with the Ricci curvature $-1$ on $\mathbb B^n$, i.e.
\begin{equation}\nonumber
\begin{aligned}
    \omega_{\mathbb B^n} &= \sqrt{-1}\partial\bar\partial \log K_{\mathbb B^n}(z,z) = -(n+1)\sqrt{-1}\partial\bar\partial \log(1-|z|^2)\\
    \end{aligned}
\end{equation}
where $K_{\mathbb B^n}$ denotes the Bergman kernel of $\mathbb B^n$.
Therefore $\log\varphi_0$ is a K\"ahler potential of $\omega_{\mathbb B^n}$.
Since $\mathcal C$ is an isometry with respect to $\omega_{\mathbb H^n}$, $\omega_{\mathbb B^n}$, 
we have 
\begin{equation}\nonumber
    \|\partial\log\varphi_0\|^2_{\omega_{\mathbb B^n}}=n+1.
\end{equation}

\subsection{Uniqueness of the constant for $\mathbb H^n$ with $n\geq 2$}\label{unique constant}
In this subsection, we will use the summation convention for duplicated indices and denote 
\begin{equation}\nonumber
    \partial_k:=\frac{\partial}{\partial w_k}, \quad
    f_j:=\frac{\partial f}{\partial w_j},\quad
    f_{\bar j}:=\frac{\partial f}{\partial \overline w_j},\quad
    f^{\bar k}:=f_jg^{j\bar k},\quad 
    f^k:=f_{\bar j}g^{k\bar j}.
\end{equation}
We will also use the covariant derivative notation using semicolon on vector coefficients $f_{j;k}$.

Let $\psi\colon \mathbb H^n\to \mathbb R$ be a function such that $\log\psi$ is a K\"ahler potential of $\omega_{\mathbb H^n}$ and satisfies $\|\partial\log\psi\|^2_{\omega_{\mathbb H^n}}\equiv c$ for some constant $c\in\mathbb R$.
Let $\Psi:=\log\psi$ and define a vector field $V$ on $\mathbb H^n$ by $$V:=\text{grad}\, \Psi = \Psi^k\partial_k.$$
Since we have 
\begin{equation}\nonumber
0=\partial_j\|\partial\Psi\|^2_{\omega_{\mathbb H^n}} = \partial_j(\Psi_i\Psi^i) = \Psi_{i;j}\Psi^i + \Psi_i \Psi^i_{\; ; j}
= \Psi_{i;j}\Psi^i + \Psi_j,
\end{equation}
we obtain 
\begin{equation}\label{tensor1}
    \Psi_{i;j}\Psi^i =- \Psi_j.
\end{equation}
Let $\text{span}\{V, \overline V\}$ denote the distribution generated by $V$ and $\overline V$.
\begin{lemma}
    $\textup{span}\{V, \overline V\}$ is integrable.
\end{lemma}
\begin{proof}
By a relation  $\Psi_{\ell ; k} = \Psi_{k ; \ell}$,
 we obtain
\begin{equation}\nonumber
\begin{aligned}
    [V,\overline V]
    &=[\Psi^k\partial_k, \Psi^{\bar j}\partial_{\bar j}] 
    = \Psi^k \Psi^{\bar j}_{\; ; k}\partial_{\bar j} - \Psi^{\bar j} \Psi^k_{\; ; \bar j}\partial_k
    =\Psi^k \Psi_{\ell ; k}g^{\ell \bar j}\partial_{\bar j} - \Psi^{\bar j} \Psi_{\bar\ell ; \bar j}g^{k\bar \ell}\partial_k\\
    &= \Psi^k \Psi_{k ; \ell}g^{\ell \bar j}\partial_{\bar j} - \Psi^{\bar j} \Psi_{\bar j ; \bar \ell}g^{k\bar \ell}\partial_k
    =-\Psi^{\bar j} \partial_{\bar j} + \Psi^k\partial_k
    =V-\overline V,
\end{aligned}
\end{equation}
    and it completes the proof.
\end{proof}

\begin{lemma}
    The integrable submanifold of $\textup{span}\{V,\overline V\}$ in $\mathbb H^n$ is totally geodesic.
\end{lemma}
\begin{proof}
Let $\nabla$ be the K\"ahler connection of $\omega_{\mathbb H^n}$.
Note that $\Psi^j_{\; ; k} = \Psi_{\bar \ell;k}g^{j\bar \ell}= g_{k \bar\ell}g^{j\bar \ell}=\delta_{jk}$. 
By~\eqref{tensor1} 
    \begin{equation}\nonumber
    \nabla_V V = \nabla_{\Psi^k\partial_k} \Psi^j\partial_j
    = \Psi^k \Psi^j_{\; ;k}\partial_j
    = \Psi^j \partial_j 
    = V
    \end{equation}
    and 
    \begin{equation}\nonumber
    \nabla_{\overline V} V = \nabla_{\Psi^{\bar k}\partial_{\bar k}} \Psi^j\partial_j
    = \Psi^{\bar k}
    \Psi^j_{\; ;\bar k}\partial_j
    = -\Psi^j \partial_j 
    = -V,
    \end{equation}
    and it implies that the integrable submanifolds are totally geodesic.
\end{proof}

Since the gradient vector field of any K\"ahler potential for a complete K\"ahler metric is complete, $V$ is complete.
Let $S$ be the maximal integrable submanifold of $\textup{span}\{V,\overline V\}$ in $\mathbb H^n$ passing through $(0,\ldots, 0,1)$.
By composing an isotropic automorphism of $\mathbb H^n$ at $(0,\ldots, 0,1)$ we may assume that $S=\{(0,\ldots, 0,w_n)\in \mathbb H^n\}\cong \mathbb H$. 
Hence, $\omega_S:= \omega_{\mathbb H^n}|_S =\frac{n+1}{2}\omega_{\mathbb H}$, i.e. $$(\partial\bar\partial \log\psi)|_S = \partial\bar\partial (\log\psi|_S)=\frac{n+1}{2}\omega_{\mathbb H}.
$$
This implies that $\log\psi|_S$ is a potential of the metric $\omega_S$.

On the other hand, since
\begin{equation}\nonumber
    \|\partial\log\psi|_S\|^2_{\omega_S}
    = \frac{\left| V\log\psi\right|^2}{\;\;\;\|V\|_{\omega_{\mathbb H}}^2}
    = \|\partial\log\psi\|_{\mathbb H^n}^2,
\end{equation}
we have 
\begin{equation}\nonumber
    \|\partial\log\psi|_S\|^2_{\omega_S}=c.
\end{equation}
Since $\omega_s$ has Ricci curvature $-\frac{2}{n+1}$, by Corollary~\ref{1-dim cor} we obtain $c=n+1$.

In summary, we obtain 
\begin{theorem}\label{uniqueness of constant}
Let $\psi\colon\mathbb H^n\to \mathbb R$ be a function such that $\log\psi$ is a K\"ahler potential of $\omega_{\mathbb H^n}$.
Suppose 
$$
\|\partial\log\psi\|_{\omega_{\mathbb H^n}}^2 \equiv c
$$
for some constant $c\in \mathbb R$.
Then $c=n+1$.
\end{theorem}

\subsection{Complete affine holomorphic vector fields.}

Let $\mathcal T_s$, $\mathcal T^{2,k}_s$, $\mathcal T^{3,k}_s$ and $\mathcal D_s$ with $k=1,\ldots, n-1$ be one-parameter families of affine automorphisms of $\mathbb H^n$ given by 
\begin{equation}\label{one parameter}
    \begin{aligned}
    \mathcal T_s(w_1, \ldots, w_n) &= (w_1, \ldots, w_{n-1}, w_n + 2\sqrt{-1}s),\\
    \mathcal T^{2,k}_s(w_1,\ldots, w_n) &= (w_1,\ldots,w_k+s, \ldots, w_{n-1}, w_n + 2s w_k + s^2), \\
    \mathcal T^{3,k}_s(w_1,\ldots, w_n) &= (w_1,\ldots, w_k + \sqrt{-1}s, \ldots, w_{n-1}, w_n - 2\sqrt{-1}s w_k + s^2),\\
    \mathcal D_s(w_1,\ldots, w_n) &= (e^s w_1,\ldots, e^s w_{n-1}, e^{2s} w_n)
    \end{aligned}
\end{equation}
for $s\in \mathbb R$. Note that 
$\mathcal C^{-1}\circ \mathcal T_s\circ \mathcal C $, 
$\mathcal C^{-1}\circ \mathcal T^{2,k}_s\circ \mathcal C $,
$\mathcal C^{-1}\circ \mathcal T_s^{3,k}\circ \mathcal C $ are parabolic automorphisms of $\mathbb B^n$ and 
$\mathcal C^{-1}\circ \mathcal D_s\circ \mathcal C $ are hyperbolic automorphisms of $\mathbb B^n$. 
Let us define automorphisms $\mathcal S^{1,k}$ of $\mathbb H^n$ by
\begin{equation}\label{permutation}
\mathcal S^{1,k}(w_1,\ldots, w_n) := (w_k,w_2,\ldots, w_{k-1}, w_1, w_{k+1},\ldots,w_n)
\end{equation}
for each $k=1,\ldots, n-1$.
These are isotropic automorphisms of $\mathbb H^n$ at $(0,\ldots, 0,1)$.

The complete holomorphic vector fields generated by $\mathcal T_s^1$, $\mathcal T_s^2$, $\mathcal T_s^3$, $\mathcal D_s$ are 
\begin{equation}\nonumber
    \begin{aligned}
    T &= 2\sqrt{-1}\frac{\partial}{\partial w_n},\\
    T^{2,k} &= 2 w_k\frac{\partial}{\partial w_n} + \frac{\partial}{\partial w_k},\quad 1\leq k \leq n-1,\\
    T^{3,k} &= -2\sqrt{-1} w_k \frac{\partial}{\partial w_n} + \sqrt{-1}\frac{\partial}{\partial w_k},\quad 1\leq k \leq n-1,\\
    D&= 2w_n\frac{\partial}{\partial w_n} + \sum_{k=1}^{n-1} w_k\frac{\partial}{\partial w_k},
    \end{aligned}
\end{equation}
respectively.
Note that $(\mathcal T_s)_* D = D-2s T$.

Let $U(n-1)\times \{1\}$ be the subgroup contained in the isotropy subgroup of $\mathbb H^n$ at $(0,\ldots,0,1)$ acting on $\mathbb H^n$ by $(w',w_n)\mapsto (Uw', w_n)$ for any $U\in U(n)$. 
The complete holomorphic vector fields generated by $U(n-1)\times \{1\}$ are of the form 
\begin{equation}\label{unitary}
\sum_{i,j=1}^{n-1} u_{ij}w_j\frac{\partial}{\partial w_i}
\end{equation}
where $u=(u_{ij})\in \{(u_{ij})_{i,j=1}^{n-1} : u_{ij}+\overline u_{ji}=0\}$, the Lie algebra of $U(n-1)$.
We will call a vector field of the form \eqref{unitary} a {\it unitary vector field}.
There are $(n-1)^2$ number of complete holomorphic vector fields corresponding to $U(n-1)\times \{1\}$:
\begin{equation}\nonumber
    U^{ij} := w_j\frac{\partial}{\partial w_i}- w_i\frac{\partial}{\partial w_j}, \quad V^{ij}:= \sqrt{-1}\left( w_j\frac{\partial}{\partial w_i} +w_i\frac{\partial}{\partial w_j} \right),\quad
     1\leq i< j\leq n-1,
\end{equation}
\begin{equation}\nonumber
W^k:= \sqrt{-1}w_k\frac{\partial}{\partial w_k},\quad 1\leq k\leq n-1.
\end{equation}

By a straightforward calculation, the pushforwards of $T$, $T^{2,k}$, $T^{3,k}$ and $D$ by the automorphisms in \eqref{one parameter} and \eqref{permutation} are as follows:
\begin{enumerate}
    \item $\mathcal{T}_s$ action:
\begin{equation}\label{P_s}
    \begin{aligned}
    &(\mathcal T_s)_*  D =  D -2s  T, \quad 
    (\mathcal T_s)_*  T = T, \quad (\mathcal T_s)_*  T^{2,k} = T^{2,k},\quad
    (\mathcal T_s)_*  T^{3,k} = T^{3,k},\\
    &(\mathcal T_s)_* U^{ij} = U^{ij},\quad
    (\mathcal T_s)_* V^{ij} = V^{ij},  
    \quad (\mathcal T_s)_* W^k = W^k
    \end{aligned}
\end{equation}
\item $\mathcal{T}^{2,k}_s$ action:
\begin{equation}\label{P_{2,k}}
\begin{aligned}
    &(\mathcal T^{2,k}_s)_* D 
    = D -s T^{2,k},
    \quad (\mathcal T^{2,k}_s)_* T  = T, 
    \quad  (\mathcal T_s^{2,k})_* T^{2,k} = T^{2,k},\\
    &(\mathcal T_s^{2,k})_* T^{3,k} 
    = 2sT+ T^{3,k},
    \quad (\mathcal T_s^{2,k})_* T^{2,\ell} 
    =  T^{2,\ell},\quad
    (\mathcal T_s^{2,k})_* T^{3,\ell} 
    =  T^{3,\ell} \,(\ell\neq k),\\
    &(\mathcal T_s^{2,k})_* U^{ij} = U^{ij}\, (k\neq i,j), 
    (\mathcal T^{2,i}_s)_*U^{ij} = U^{ij} + sT^{2,j},\quad 
    (\mathcal T^{2,j}_s)_*U^{ij} = U^{ij} -sT^{2,i}, \\
    &(\mathcal T_s^{2,k})_* V^{ij} = V^{ij}\, (k\neq i,j),\quad
    (\mathcal T^{2,i}_s)_*V^{ij} = V^{ij} - sT^{3,j},\quad 
(\mathcal T^{2,j}_s)_*V^{ij} = V^{ij} -sT^{3,i}, \\
& (\mathcal T^{2,k}_s) W^k =  W^k -sT^{3,k} - s^2 T,\quad (\mathcal T^{2,k}_s) W^\ell = W^\ell \,( k\neq\ell),
    \end{aligned}
\end{equation}
\item $\mathcal{T}^{3,k}_s$ action:
\begin{equation}\label{P_{3,k}}
\begin{aligned}
    &(\mathcal T^{3,k}_s)_* D 
    = D -s T^{3,k},
     \quad (\mathcal T^{3,k}_s)_* T  
    = T, 
     \quad (\mathcal T_s^{3,k})_* T^{2,k} = -2sT+T^{2,k},\\
    &(\mathcal T_s^{3,k})_* T^{3,k} 
    =T^{3,k},
    \quad (\mathcal T_s^{3,k})_* T^{2,\ell} 
    = T^{2,\ell},\quad 
    (\mathcal T_s^{3,k})_* T^{3,\ell} 
    =  T^{3,\ell}\,(\ell\neq k),\\
    &(\mathcal T_s^{3,k})_* U^{ij} = U^{ij} \,(k\neq i,j)\quad (\mathcal T^{3,i}_s)_*U^{ij} = U^{ij} + sT^{3,j},\quad 
    (\mathcal T^{3,j}_s)_*U^{ij} = U^{ij} -sT^{3,i},\\
    &(\mathcal T_s^{3,k})_* V^{ij} = V^{ij} \,(k\neq i,j),\quad
    (\mathcal T^{3,i}_s)_*V^{ij} = V^{ij} + sT^{2,j},\quad 
    (\mathcal T^{3,j}_s)_*V^{ij} = V^{ij} + sT^{2,i},\\
    &(\mathcal T^{3,k}_s) W^k =  W^k + sT^{2,k} - s^2 T,\quad 
    (\mathcal T^{3,k}_s) W^\ell = W^\ell \,(k\neq\ell),
    \end{aligned}
\end{equation} 
\item $\mathcal S^{1,k}$ action:
\begin{equation}\label{S}
\begin{aligned}
    &(\mathcal S^{1,k})_* T^{2,k} = T^{2,1},\quad
    (\mathcal S^{1,k})_* T^{3,k} = T^{3,1},\quad
    (\mathcal S^{1,k})_* T = T,\\
    &(\mathcal S^{1,k})_* T^{2,\ell}=T^{2,\ell},\quad
    (\mathcal S^{1,k})_* T^{3,\ell}=T^{3,\ell}\,(\ell\neq k),\\
    & (\mathcal S^{1,i})_* U^{ij} = U^{1j},\quad 
    (\mathcal S^{1j})_* U^{1j} = -U^{1j},\\
    & (\mathcal S^{1,i})_* V^{ij} = V^{1j},\quad 
    (\mathcal S^{1,j})_* V^{1j} = V^{1j},\quad\\
    &(\mathcal S^{1,k})_*W^k = W^1,\quad
    (\mathcal S^{1,k})_* W^\ell = W^\ell\,(k\neq \ell).
\end{aligned}
\end{equation}
\end{enumerate}

\subsection{Condition to have constant 
 differential norm}
Let $\psi\colon\mathbb H^n\to \mathbb R$ be a function such that $$
\sqrt{-1}\partial\bar\partial\log\psi = \omega_{\mathbb H^n} \quad\text{ and }\quad \|\partial\log \psi\|_{\omega_{\mathbb H^n}}^2 \equiv c
$$
for some constant $c$.
Let $f$ be a holomorphic function such that $\log \psi = \log\psi_0 + f+\overline f$.
By Theorem~\ref{uniqueness of constant}, we obtain $c= n+1$.
Since $\|\partial\log \psi\|_{\omega_{\mathbb H^n}}^2=\|\partial\log \psi_0\|_{\omega_{\mathbb H^n}}^2  = n+1$, we have
\begin{equation}\nonumber
    \begin{aligned}
    n+1&= \| \partial\log\psi\|^2_{\omega_{\mathbb H^n}} = \sum_{k,j=1}^n \frac{\partial \log\psi}{\partial w_k}\frac{\partial \log\psi}{\partial \overline w_j} g^{k\overline j}\\
    & = \|\partial\log \psi_0\|_{\omega_{\mathbb H^n}}^2
    + \sum_{k,j=1}^n \frac{\partial f}{\partial w_k} \frac{\partial\log\psi_0}{\partial \overline w_j} g^{k\overline j}
    + \sum_{k,j=1}^n \frac{\partial \overline f}{\partial \overline w_j} \frac{\partial\log\psi_0}{\partial w_k} g^{k\overline j}  
    + \sum_{k,j=1}^n \frac{\partial f}{\partial w_k} \frac{\partial \overline f}{\partial \overline w_j} g^{k\overline j},
    \end{aligned}
\end{equation}
and as a result we obtain
\begin{equation}\label{eq_constant norm}
\begin{aligned}
    0&=2\textup{Re} \sum_{k,j=1}^n \frac{\partial \overline f}{\partial \overline w_j} \frac{\partial\log\psi_0}{\partial w_k} g^{k\overline j}  
    + \sum_{k,j=1}^n \frac{\partial f}{\partial w_k} \frac{\partial \overline f}{\partial \overline w_j} g^{k\overline j}\\
    &= -4(\text{Re}\,w_n -|w'|^2)\text{Re}\,\frac{\partial f}{\partial w_n} + \|df\|_{\omega_{\mathbb H^n}}^2\\
    &=-4(n+1)\text{Re}\,\frac{\partial f}{\partial w_n} + \sum_{j=1}^{n-1}\left|\frac{\partial f}{\partial w_j}\right|^2 + \text{Re}\left( \frac{\partial \overline f}{\partial\overline w_n}\sum_{j=1}^{n-1}\frac{\partial f}{\partial w_j}  w_j\right)
    +4\text{Re}\,w_n\left|\frac{\partial f}{\partial w_n}\right|^2.
    \end{aligned}
\end{equation}

It is worth to note that by \eqref{eq_constant norm} $\frac{\partial f}{\partial w_n}\colon \mathbb H^n\to \mathbb C$ is a holomorphic function of which image is contained in $\overline{\mathbb H^1}=\{w\in \mathbb C: \text{Re}\,w\geq 0\}$.
Considering the holomorphic function $\frac{\partial f}{\partial w_n}\circ \mathcal C\colon \mathbb B^n\to \mathbb C$, if there exists a point in $\mathbb H^n$ such that $\text{Re}\frac{\partial f}{\partial w_n}=0$, then $\text{Re}\frac{\partial f}{\partial w_n}$ vanishes identically.
By the second equation in \eqref{eq_constant norm} $df\equiv 0$ on $\mathbb H^n$ and hence $f$ is a constant.
\begin{lemma}\label{maximum principle}
    Let $\psi\colon\mathbb H^n\to \mathbb R$ be a K\"ahler potential of $\omega_{\mathbb H^n}$ which has constant differential norm with respect to the Bergman metric. Let $f$ be a holomorphic function on $\mathbb H^n$ such that $\log\psi =\log\psi_0 + f+\overline f$. If there exists a point $p\in \mathbb H^n$ such that $\textup{Re}\,\frac{\partial f}{\partial w_n}(p)=0$, then $f$ is a constant.
\end{lemma}

\subsection{Hyperbolic vector fields}\label{hyperbolic}
Since we have
\begin{equation}\nonumber
    \begin{aligned}
     D \log \psi_0 
    &=-(n+1)\left( 2w_n\frac{\partial}{\partial w_n} + \sum_{k=1}^{n-1} w_k\frac{\partial}{\partial w_k}\right) \log\left(\textup{Re } w_n - |w'|^2 \right)\\
    &=-\frac{n+1}{\text{Re}\,w_n-|w'|^2}\left( 2w_n\frac{\partial}{\partial w_n} + \sum_{k=1}^{n-1} w_k\frac{\partial}{\partial w_k}\right) \left(\textup{Re } w_n - |w'|^2 \right)\\
    &=-\frac{n+1}{\text{Re}\,w_n-|w'|^2}
    \left( w_n - |w'|^2\right),
    \end{aligned}
\end{equation}
we obtain \begin{equation}\label{H phi_0}
    (\textup{Re } D) \log \psi_0 = -(n+1).
\end{equation}
Let $U=\sum_{i,j=1}^{n-1} u_{ij}w_j\frac{\partial}{\partial w_i}$ be a unitary vector field.
Since we have 
\begin{equation}\nonumber
\begin{aligned}
U\log\psi_0 &= -\frac{n+1}{\text{Re}\,w_n-|w'|^2}\left(\sum_{i,j=1}^{n-1} u_{ij}w_j\frac{\partial}{\partial w_i}\right)(\text{Re}\,w_n-|w'|^2)\\
&=\frac{n+1}{\text{Re}\,w_n-|w'|^2}\sum_{i,j=1}^{n-1} u_{ij}w_j\overline w_i\\
&=\frac{n+1}{\text{Re}\,w_n-|w'|^2}\left(
\sum_{i<j}^{n-1} u_{ij}w_j\overline w_i
+\sum_{i>j}^{n-1} u_{ij}w_j\overline w_i
\right)\\
&=\frac{2\sqrt{-1}(n+1)}{\text{Re}\,w_n-|w'|^2}\text{Im}
\left(\sum_{i<j}^{n-1} u_{ij}w_j\overline w_i\right),
\end{aligned}
\end{equation}
we obtain
\begin{equation}\label{U phi_0}
(\text{Re}\,U)\log\psi_0 
=0.
\end{equation}
\begin{lemma}\label{H_lemma}
Let $\psi\colon \mathbb H^n\to \mathbb R$ be a function with $\partial\overline\partial \log \psi = \omega_{\mathbb H^n}$ and $\|\partial\log \psi \|^2_{\mathbb H^n} \equiv n+1$. Let $U$ be a unitary vector field.
If $$(\textup{Re}\, (D + U))\log \psi \equiv c$$ for some constant $c$, then $c = \pm (n+1)$.
\end{lemma}
\begin{proof}
Since $\sqrt{-1}\partial\bar\partial \log \psi_0 = \sqrt{-1}\partial\bar\partial\log\psi = \omega_{\mathbb H^n}$, $\log\psi - \log\psi_0$ is a pluriharmonic function. Let $f\colon \mathbb H^n\to\mathbb C$ be a holomorphic function such that $\log\psi - \log\psi_0 = f+\bar f$.
As a result we have 
\begin{equation}\nonumber
    (\textup{Re} (D+ U)) \log \psi
    = -(n+1) + \textup{Re}\left( 2w_n\frac{\partial f}{\partial w_n} + \sum_{k=1}^{n-1} w_k\frac{\partial f}{\partial w_k}
    + \sum_{i,j=1}^{n-1} u_{ij}w_j\frac{\partial f}{\partial w_i}
    \right)\equiv c
\end{equation}
by \eqref{H phi_0} and \eqref{U phi_0}, and hence
 $w_n\frac{\partial f}{\partial w_n}$ restricted to $\{(0,\ldots,0, w_n)\in \mathbb H^n\}$ is constant. Let 
\begin{equation}\label{f'}
\frac{\partial f}{\partial w_n}(0,\ldots, 0,w_n) = \frac{C}{w_n}
\end{equation}
for some constant $C$.
By \eqref{eq_constant norm} 
we obtain
\begin{equation}\label{hyperbolic eqn}
0=-4(n+1)\text{Re} \frac{C}{w_n} + \sum_{j=1}^{n-1}\left| \frac{\partial f}{\partial w_j}\right|^2 + 4\text{Re}\,w_n \left|\frac{C}{w_n}\right|^2.
\end{equation}
Write $C=\alpha+\sqrt{-1}\beta$ with $\alpha,\,\beta\in\mathbb R$. Suppose $\beta\neq 0$. 
Then for $w_n=-\sqrt{-1}\beta$, we obtain a contradiction to the equation \eqref{hyperbolic eqn} and hence $C$ is a real number.
As a result, we obtain 
\begin{equation}\nonumber
    0=-4(n+1)C \frac{\text{Re}\,w_n}{|w_n|^2} + \sum_{j=1}^{n-1}\left| \frac{\partial f}{\partial w_j}\right|^2 + 4\text{Re}\,w_n \left|\frac{C}{w_n}\right|^2.
\end{equation}
By the maximum principle, $\frac{\partial f}{\partial w_j}\equiv 0$ on $\{(0,\ldots,0,w_n)\}$ for each $j=1,\ldots, n-1$.
Hence, we obtain
$$
(n+1)C =C^2
$$
 and hence $C=0$ or $C=n+1$.
Since $c=2C-(n+1)$ by \eqref{f'}, $c = -(n+1)$ or $c=n+1$.
\end{proof}

\subsection{Parabolic vector fields}\label{parabolic}

Since 
$$
T\log\psi_0 
= -
(n+1)\left( 2\sqrt{-1}\frac{\partial}{\partial w_n}\right)\log \left(\textup{Re } w_n - |w'|^2 \right)
= \frac{- \sqrt{-1}(n+1)}{\textup{Re } w_n - |w'|^2},
$$
we have 
\begin{equation}\label{P_1 log phi_0}
    (\textup{Re }T )\log\psi_0  \equiv 0.
\end{equation}
\begin{lemma}\label{P_lemma}
Let $\psi\colon \mathbb H^n\to\mathbb R$ be a function with $\partial\bar\partial\log\psi = \omega_{\mathbb H^n}$ and $\|\partial\log\psi\|^2_{\mathbb H^n}\equiv n+1$. Let $U$ be a unitary vector field. If 
$$ 
(\textup{Re}\,(T+U)) \log \psi \equiv 
    c$$ for some constant $c$, then $c=0$ and $\psi= r\psi_0$ for some $r>0$.
\end{lemma}
\begin{proof}
Let $f$ be a holomorphic function on $\mathbb H^n$ such that $\log \psi=\log\psi_0+f+\overline f$.
Since by \eqref{P_1 log phi_0}
\begin{equation}\label{p_1 log phi}
\begin{aligned}
    (\textup{Re}\,(T+U))\log\psi = \textup{Re} \,(Tf +Uf)
    = \textup{Re}\left( 2i\frac{\partial f}{\partial w_n}
    +\sum_{i,j=1}^{n-1}u_{ij}w_j\frac{\partial f}{\partial w_i}\right)\equiv c,
    \end{aligned}
\end{equation}
and this implies that $\frac{\partial f}{\partial z_n}\equiv C$ for some constant $C$ on $\{(0,\ldots,0,w_n)\in\mathbb H^n\}$. Hence, by \eqref{eq_constant norm} we obtain 
\begin{equation}\label{parabolic constant norm eq}
    \begin{aligned}
    0&=-4(n+1)\text{Re}\,C 
    + \sum_{j=1}^{n-1}\left|\frac{\partial f}{\partial w_j}\right|^2 
    +4|C|^2\text{Re}\,w_n\\
    & =\sum_{j=1}^{n-1}\left|\frac{\partial f}{\partial w_j}\right|^2 +\text{Re}\left(
    -4(n+1)C  +4|C|^2 w_n
    \right)
    \end{aligned}
\end{equation}
on $\{(0,\ldots,0,w_n)\in\mathbb H^n\}$.
By differentiating \eqref{parabolic constant norm eq} with respect to $\frac{\partial^2}{\partial w_n\partial \overline w_n}$ we have 
\begin{equation}\nonumber
    \sum_{j=1}^{n-1}\left| \frac{\partial^2f}{\partial w_n\partial w_j}\right|^2\equiv 0
\end{equation}
and as a result $\frac{\partial f}{\partial w_j}$ with $j=1,\ldots,n-1$ are constant on $\{(0,\ldots,0,w_n)\in\mathbb H^n\}$. By \eqref{parabolic constant norm eq}  we have $4|C|^2 w_n$ is constant on $\{(0,\ldots,0,w_n)\in \mathbb H^n\}$. Therefore $C=0$ and $c=0$.
Thus  $f$ is a constant by \eqref{p_1 log phi} and \eqref{parabolic constant norm eq}. This completes the proof.
\end{proof}
\begin{remark}
For the vector fields $T^{2,k}$ and $T^{3,k}$ with $k=1,\ldots, n-1$ we also have 
\begin{equation}\nonumber
    (\textup{Re }T^{2,k}) \log \psi_0
    =(\textup{Re }T^{3,k}) \log \psi_0\equiv 0
    \end{equation}
since
\begin{equation}\nonumber
\begin{aligned}
T^{2,k} \log \psi_0 
&= -(n+1)\left(2 w_k\frac{\partial}{\partial w_n} + \frac{\partial}{\partial w_k}\right) \log \left(\textup{Re } w_n - |w'|^2 \right)\\
&=\frac{-(n+1)2\sqrt{-1}\,\text{Im }w_k}{\textup{Re } w_n - |w'|^2},
\end{aligned}
\end{equation}
and 
\begin{equation}\nonumber
\begin{aligned}
T^{3,k} \log \psi_0 
&= -(n+1)\left(-2\sqrt{-1} w_k \frac{\partial}{\partial w_n} + \sqrt{-1}\frac{\partial}{\partial w_k}\right) \log \left(\textup{Re } w_n - |w'|^2 \right)\\
&=\frac{(n+1)2\sqrt{-1}\,\text{Re}\,w_k}{\textup{Re } w_n - |w'|^2}.
\end{aligned}
\end{equation}
\end{remark}

\section{Proof of Theorem~\ref{main1}}\label{proof of main1}
Let $\psi\colon \mathbb H^n\to\mathbb R$ be a positive valued function such that 
$$ 
\sqrt{-1}\partial\bar\partial\log\psi = \omega_{\mathbb H^n}, \quad \|\partial\log \psi\|_{\omega_{\mathbb H^n}}^2=n+1.
$$
Let 
\begin{equation}\nonumber
W:=\sqrt{-1} \psi^{\frac{-n}{n+1}}\textup{grad}(\psi)
\end{equation}
which is the nowhere vanishing holomorphic complete vector field on $\mathbb H^n$ given in \cite[Theorem 3.2]{Choi-Lee2021} that satisfies $(\textup{Re}\,W)\log\psi \equiv 0$.

By composing an isotropy of $\mathbb H^n$ at $(0,\ldots,0,1)$ we can assume that $W$ vanishes at infinity, i.e. $\mathcal F_* W$ vanishes at $(0,\ldots,0,1)\in\partial\mathbb B^n$.
Since the Lie algebra of the automorphism group of $\mathbb H^n$, i.e. the set of complete holomorphic vector fields on $\mathbb H^n$ is isomorphic to $\mathfrak{su}(n,1)$, its dimension is $n^2+2n$. Hence it is generated by the holomorphic vector fields corresponding to $\mathfrak{s(u(1)\times u(n))}$ and $\mathcal D$, $\mathcal T^1$, $\mathcal T^{2,k}$, $\mathcal T^{3,k}$ with $k=1,\ldots,n-1$.
However, the complete holomorphic vector fields that vanish at infinity are  $D$, $T$, $T^{2,k}$, $T^{3,k}$, $U^{ij}$ and $V^{ij}$, 
$W$ has a decomposition 
$$
W = aD + b T + \sum_{k=1}^{n-1} c_k T^{2,k} + \sum_{k=1}^{n-1} d_k T^{3,k} +\sum_{1\leq i<j\leq n-1}e_{ij}U^{ij}
+\sum_{1\leq i<j\leq n-1}f_{ij}V^{ij}
+\sum_{k=1}^{n-1}g_kW^k
$$
with $a,b,c_k, d_k, e_{ij}, f_{ij}, g_k\in \mathbb R$.
Note that  $b\neq 0$ or $\sum_{k=1}^{n-1}c_k\neq0$ or $\sum_{k=1}^{n-1}d_k\neq0$ by Lemma~\ref{H_lemma}.

\smallskip
{\bf Step 1 :} By~\eqref{S} we have 
\begin{equation}\nonumber
\left(\mathcal S^{1,n-1}\mathcal S^{1,n-2}\cdots\mathcal S^{1,2}\right)_* W =  aD + b T 
+ c T^{2,1} 
+ dT^{3,1}
+\sum_{j=2}^{n-1} e_jU^{1j}
+\sum_{j=2}^{n-1}f_jV^{1j}
+gW^1
\end{equation}
where $c:=\sum_{k=1}^{n-1} c_k$, 
$d:=\sum_{k=1}^{n-1} d_k $,
$e_j:=-\sum_{i=1}^{j-1}e_{ij}$,
$f_j:=\sum_{i=1}^{j-1}f_{ij}$
and $g:=\sum_{k=1}^{n-1}g_k$.

For $\mathcal T:=\mathcal T^{3,n-1}_{s_{3,n-1}}\cdots \mathcal T^{3,1}_{s_{3,1}} \mathcal T^{2,n-1}_{s_{2,n-1}}\cdots \mathcal T^{2,1}_{s_{2,1}} \mathcal T_{s_1}$ 
by~\eqref{P_s}, \eqref{P_{2,k}},  \eqref{P_{3,k}} we have
\begin{equation}\nonumber
 \mathcal T_*  D
= D + \left(2\sum_{j=1}^ns_{2,j}s_{3,j} - 2s_1\right)T -\sum_{j=1}^ns_{3,j}T^{3,j} - \sum_{j=1}^n s_{2,j} T^{2,j},
\end{equation}
\begin{equation}\nonumber
\begin{aligned}
\mathcal T_*   T
&= T,\\
\mathcal T_*  T^{2,1}
&= T^{2,1} - 2s_{3,1}T,\\
\mathcal T_*  T^{3,1}
&=T^{3,1} + 2s_{2,1}T,
\end{aligned}
\end{equation}
\begin{equation}\nonumber
\begin{aligned}
&\mathcal T_*  U^{1j}
= U^{1j} + 2(s_{3,1}s_{2,j} - s_{2,1}s_{3,j})T - s_{3,j}T^{3,1} - s_{2,j}T^{2,1} + s_{3,1} T^{3,j} + s_{2,1}T^{2,j},
\\
&\mathcal T_*  V^{1j}
= V^{1j} - 2(s_{3,1}s_{3,j} + s_{2,1}s_{2,j})T + s_{3,j}T^{2,1} - s_{2,j}T^{3,1} + s_{3,1} T^{2,j} - s_{2,1}T^{3,j},
\end{aligned}
\end{equation}
and 
\begin{equation}\nonumber
\begin{aligned}
&\mathcal T_* W^1
= W^1 + s_{3,1}T^{2,1} - s_{2,1}T^{3,1} - (s_{3,1}^2 +s_{2,1}^2)T.
\end{aligned}
\end{equation}
Therefore 
\begin{equation}\label{fixed W}
\begin{aligned}
&\left(\mathcal T\mathcal S^{1,n-1}\mathcal S^{1,n-2}\cdots\mathcal S^{1,2}\right)_* W \\
&= aD 
+\left(2a\sum_{j=1}^ns_{2,j}s_{3,j} 
- 2as_1+b -2cs_{3,1} +2ds_{2,1}\right.\\
&\quad\quad\left.+\sum_{j=2}^{n-1}2e_j(s_{3,1}s_{2,j} - s_{2,1}s_{3,j})
- \sum_{j=2}^{n-1}2f_j(s_{3,1}s_{3,j} + s_{2,1}s_{2,j})- (s_{3,1}^2 +s_{2,1}^2)g\right)T\\
&
+\left( -as_{2,1}+c -\sum_{j=2}^{n-1} e_j
s_{2,j} +\sum_{j=2}^{n-1} f_js_{3,j} + gs_{3,1}\right)T^{2,1} \\
&+\left(-as_{3,1}+d -\sum_{j=2}^{n-1} e_j s_{3,j} -\sum_{j=2}^{n-1}f_js_{2,j}
 -s_{2,1}g\right)T^{3,1}\\ 
 &+\sum_{j=2}^{n-1}(-as_{2,j}+s_{2,1}e_j + s_{3,1}f_j)T^{2,j}
 +\sum_{j=2}^{n-1}(-as_{3,j}+s_{3,1}e_j - s_{2,1}f_j)T^{3,j}\\
 &+\sum_{j=2}^{n-1} e_jU^{1j}
+\sum_{j=2}^{n-1}f_jV^{1j}
+gW^1.
\end{aligned}
\end{equation}
Consider the following system of linear equations:
\begin{equation}\label{system of eqns}
\begin{aligned}
 &-as_{2,1}+c -\sum_{j=2}^{n-1} e_j
s_{2,j} +\sum_{j=2}^{n-1} f_js_{3,j} + gs_{3,1}=0,\\
&-as_{3,1}+d -\sum_{j=2}^{n-1} e_j s_{3,j} -\sum_{j=2}^{n-1}f_js_{2,j}
 -s_{2,1}g=0,\\
 &-as_{2,j}+s_{2,1}e_j + s_{3,1}f_j=0, \quad\forall j=2,\ldots,n-1\\
 &-as_{3,j}+s_{3,1}e_j - s_{2,1}f_j=0 ,\quad\forall j=2,\ldots,n-1
\end{aligned}
\end{equation}
which makes the equation~\eqref{fixed W} not to have $T^{2,j}$, $T^{3,j}$, $j=1,\dots, n-1$ factors.
This system of linear equations equals to 
\begin{equation}\label{matrix}
    \left(\begin{array}{cc|cc|c|cc}
    a&-g&e_2&-f_2&\cdots&e_{n-1}&-f_{n-1}\\
    g&a&f_2&e_2&\cdots&f_{n-1}&e_{n-1}\\\hline
    -e_2&-f_2& a&0 &&&\\
    f_2& -e_2&0&a&&&\\\hline
    \vdots &\vdots&&&\ddots&&\\\hline
    -e_{n-1}&-f_{n-1} &&&&a&0\\
    f_{n-2}&-e_{n-2}& &&&0&a
    \end{array}\right)
    \left(\begin{array}{c}
    s_{2,1}\\
    s_{3,1}\\
    s_{2,2}\\
    s_{3,2}\\
    \vdots\\
    s_{2,n-1}\\
    s_{3,n-1}
    \end{array}\right) 
    =\left(\begin{array}{c}
    c\\
    d\\
    0\\
    0\\
    \vdots\\
    0\\
    0
    \end{array}\right).
\end{equation}
Since $a\neq 0$ implies that the determinant of the left side matrix in \eqref{matrix} is nonzero, there exists $s_{2,\ell}$ and $s_{3,\ell}$ which satisfy the equation~\eqref{matrix}.
Moreover we can choose $s_1$ such that the equation \eqref{fixed W} has no $T$ factor.
With such $s_{2,\ell}$, $s_{3,\ell}$ and $s_1$, for
$$
\widetilde W := \left( \mathcal T^{3,n-1}_{s_{3,n-1}}\cdots \mathcal T^{3,1}_{s_{3,1}} \mathcal T^{2,n-1}_{s_{2,n-1}}\cdots \mathcal T^{2,1}_{s_{2,1}} \mathcal T_{s_1}\mathcal S^{1,n-1}\mathcal S^{1,n-2}\cdots\mathcal S^{1,2}\right)_* W$$
and $$
\widetilde\varphi := \varphi\circ \left( \mathcal T^{3,n-1}_{s_{3,n-1}}\cdots \mathcal T^{3,1}_{s_{3,1}} \mathcal T^{2,n-1}_{s_{2,n-1}}\cdots \mathcal T^{2,1}_{s_{2,1}} \mathcal T_{s_1}\mathcal S^{1,n-1}\mathcal S^{1,n-2}\cdots\mathcal S^{1,2}\right)
$$
we obtain 
\begin{equation}\nonumber
    \widetilde W \log\widetilde\psi = \left(aD
    -\sum_{j=2}^{n-1}e_jU^{1j}
    +\sum_{j=2}^{n-1}f_jV^{1j}+ gW^1\right)
    \log\widetilde\psi.
\end{equation}
Since $0=(\textup{Re } \widetilde W)\log\widetilde \psi=a(\textup{Re }D)\log\widetilde \psi = (n+1)a\log\widetilde\psi$ by Lemma~\eqref{H_lemma}, we obtain $a=0$.

\smallskip
{\bf Step 2:} Since $a=0$, we have 
\begin{equation}\nonumber
\left(\mathcal S^{1,n-1}\mathcal S^{1,n-2}\cdots\mathcal S^{1,2}\right)_* W =   b T 
+ c T^{2,1} 
+ dT^{3,1}
+\sum_{j=2}^{n-1} e_jU^{1j}
+\sum_{j=1}^{n-1}f_jV^{1j}
+gW^1.
\end{equation}

If $c$ and $d$ are zero, then by Lemma~\ref{P_lemma} the proof is completed.

Now assume $c$ or $d$ is nonzero. Moreover by multiplying $-1$ if $b$ is nonpositive, we may assume that $b\geq 0$.
By~\eqref{fixed W} we have
\begin{equation}\nonumber
\begin{aligned}
&\left(\mathcal T\mathcal S^{1,n-1}\mathcal S^{1,n-2}\cdots\mathcal S^{1,2}\right)_* W  \\
&
=\left(b-2cs_{3,1} +2ds_{2,1}+\sum_{j=2}^{n-1}2e_j(s_{3,1}s_{2,j} - s_{2,1}s_{3,j})
- \sum_{j=2}^{n-1}2f_j(s_{3,1}s_{3,j} + s_{2,1}s_{2,j})- (s_{3,1}^2 +s_{2,1}^2)g\right)T\\
&
+\left( c -\sum_{j=2}^{n-1} e_j
s_{2,j} +\sum_{j=2}^{n-1} f_js_{3,j} + gs_{3,1}\right)T^{2,1} 
+\left(d -\sum_{j=2}^{n-1} e_j s_{3,j} -\sum_{j=2}^{n-1}f_js_{2,j}
 -s_{2,1}g\right)T^{3,1}\\ 
 &+\sum_{j=2}^{n-1}(s_{2,1}e_j + s_{3,1}f_j)T^{2,j}
 +\sum_{j=2}^{n-1}(s_{3,1}e_j - s_{2,1}f_j)T^{3,j}
 +\sum_{j=2}^{n-1} e_jU^{1j}
+\sum_{j=2}^{n-1}f_jV^{1j}
+gW^1.
\end{aligned}
\end{equation}

If $e_j=f_j=0$ for all $j$ and $g\neq 0$.
Then by choosing $s_{3,1}=-\frac{c}{g}$ and $s_{2,1}=\frac{d}{g}$,
\begin{equation}\nonumber
\left(\mathcal T\mathcal S^{1,n-1}\mathcal S^{1,n-2}\cdots\mathcal S^{1,2}\right)_* W = \left(b+\frac{c^2+d^2}{g}\right)T +gW^1
\end{equation}
for some $\alpha\in \mathbb R$ and by Lemma~\ref{P_lemma} the proof is completed

If $e_j=f_j=0$ for all $j$ and $g= 0$.
Then by choosing $s_{3,1}$, $s_{2,1}$ such that $b-2cs_{3,1}+2ds_{2,1}=0$, we have  
$$
\left( \mathcal TS^{1,n-1}\mathcal S^{1,n-2}\cdots\mathcal S^{1,2}\right)_* W
    = cT^{2,1} + dT^{3,1}.
$$
Let
$$
\widetilde W :=\left( \mathcal T^{3,n-1}_{s_{3,n-1}}\cdots \mathcal T^{3,1}_{s_{3,1}} \mathcal T^{2,n-1}_{s_{2,n-1}}\cdots \mathcal T^{2,1}_{s_{2,1}} \mathcal S^{1,n-1}\mathcal S^{1,n-2}\cdots\mathcal S^{1,2}\right)_* W
$$
and 
$$
\widetilde\psi := \psi\circ \left( \mathcal T^{3,n-1}_{s_{3,n-1}}\cdots \mathcal T^{3,1}_{s_{3,1}} \mathcal T^{2,n-1}_{s_{2,n-1}}\cdots \mathcal T^{2,1}_{s_{2,1}} \mathcal S^{1,n-1}\mathcal S^{1,n-2}\cdots\mathcal S^{1,2}\right) .
$$
Then we obtain 
\begin{equation}\nonumber
    \widetilde W = 2Aw_1\frac{\partial }{\partial w_n} 
 +\overline A \frac{\partial }{\partial w_1}
\end{equation}
for $A=c-\sqrt{-1}d$. We remark that $A$ is nonzero.
Let $f$ be a holomorphic function such that $\log\tilde\psi =\log \psi_0 + f+\overline f$.
By the relation $0=(\textup{Re}\,\widetilde W)\log\widetilde\psi$, we obtain 
\begin{equation}\label{tilde W log psi}
2Aw_1\frac{\partial f }{\partial w_n} 
 +\overline A \frac{\partial f}{\partial w_1}=0
\end{equation}
and as a result \begin{equation}\label{fw1=0} 
\frac{\partial f}{\partial w_1}(0,w_2,\ldots, w_n)=0.
\end{equation}
By differentiating \eqref{tilde W log psi} with respect to $\frac{\partial}{\partial w_1}$, we obtain
\begin{equation}\label{differentiate}
2A\frac{\partial f }{\partial w_n} 
+2Aw_1\frac{\partial^2 f }{\partial w_1\partial w_n}
 +\overline A \frac{\partial^2 f}{\partial w_1^2}=0.
\end{equation}
By differentiating \eqref{eq_constant norm}
with respect to $\frac{\partial^2}{\partial w_1\partial\overline w_1}$
we obtain
\begin{equation}\label{w1w1}
\begin{aligned}
    0&= \sum_{j=1}^{n-1}\left|\frac{\partial^2 f}{\partial w_1\partial w_j}\right|^2 
    + \text{Re}\left( \frac{\partial^2 \overline f}{\partial \overline w_1\partial\overline w_n}\sum_{j=1}^{n-1}\frac{\partial}{\partial w_1}\left(\frac{\partial f}{\partial w_j}  w_j\right)\right)
    +4\text{Re}w_n\left|\frac{\partial^2 f}{\partial w_1\partial w_n}\right|^2.
\end{aligned}
\end{equation}
At any point $(0,0,\ldots,0,w_n)\in \mathbb H^n$, by~\eqref{fw1=0} and \eqref{w1w1} we obtain
\begin{equation}\nonumber
\begin{aligned}
    0&= \sum_{j=1}^{n-1}\left|\frac{\partial^2 f}{\partial w_1\partial w_j}\right|^2 
    +4\text{Re}\,w_n\left|\frac{\partial^2 f}{\partial w_1\partial w_n}\right|^2.
\end{aligned}
\end{equation}
In particular, we obtain 
\begin{equation}\nonumber
\frac{\partial^2 f}{\partial w_1^2}(0,\ldots,0,w_n)=0.
\end{equation}
By \eqref{differentiate}
\begin{equation}
\frac{\partial f}{\partial w_n}(0,\ldots,0,w_n)=0
\end{equation}
for any $w_n$ with $\text{Re}\,w_n>0$ and hence by Lemma~\ref{maximum principle} $f$ is a constant. 
This completes the proof.

If $e_j\neq 0$ or $f_j\neq 0$ for some $j$, then let $s_{2,1}=s_{3,1}=0$, $s_{2,\ell}=0$, $s_{3,\ell}=0$ for $\ell\neq j$ and 
$$
s_{2,j}=\frac{e_jc + f_jd}{e_j^2 + f_j^2}, \quad 
s_{2,j}=\frac{-f_jc + e_j d}{e_j^2 + f_j^2}
$$
so that  $c=\sum_{j=2}^{n-1}e_js_{2,j} + \sum_{j=2}^{n-1} f_js_{3,j}$ and $d=\sum_{j=2}^{n-1}e_js_{3,j} + \sum_{j=2}^{n-1} f_js_{2,j}$. As a result 
\begin{equation}
\left(\mathcal T\mathcal S^{1,n-1}\mathcal S^{1,n-2}\cdots\mathcal S^{1,2}\right)_* W = bT+\sum_{j=2}^{n-1} e_jU^{1j}
+\sum_{j=2}^{n-1}f_jV^{1j}
+gW^1.
\end{equation}
By Lemma~\ref{P_lemma}, the proof is completed. \qed

\section{Lie algebra structure of $\mathbb H^n$}\label{Lie algebra structure of Hn}
In \cite{KMO1970} Kaup--Matsushima--Ochiai studied the Lie algebra structures of homogeneous Siegel domains of the second kind. The goal of this section is to find a basis of the Lie algebra $\aut(\mathbb H^n)$ with respect to the decomposition which Kaup--Matsushima--Ochiai described in their work.

For a domain $\Omega$ in $\mathbb C^n$, denote by $\text{Aff}(\Omega)$ the group of affine transformations of $\Omega$, i.e. $$\text{Aff}(\Omega) = \{ f: f(z) = Az + b, A\in GL(n,\mathbb C), b\in \mathbb C^n, f(\Omega)=\Omega\}$$
and denote by $\aff(\Omega)$ its Lie algebra, which is isomorphic to the set of complete holomorphic vector fields on $\Omega$.

By \cite[Theorem 4]{KMO1970} the Lie algebra $\aut(\mathbb H^n)$ of $\Aut(\mathbb H^n)$ has a decomposition 
\begin{equation*}
\aut(\mathbb H^n) = \g_{-1}\oplus\g_{-1/2}\oplus\g_{0}\oplus\g_{1/2}\oplus\g_{1} 
\end{equation*}
where $\g_a=\{V\in\aut(\mathbb H^n):[D,V]=aV \}$.  
Then we have 
$$
T\in \g_{-1},\quad
T^{2,k}, T^{3,k}\in\g_{-1/2},\quad
D,U^{ij}, V^{ij}, W^k\in \g_0.
$$
Let us consider the involutive automorphism $\sigma$ of $\mathbb H^n$ at $(0,\ldots,0,1)$, that is, it satisfies $\sigma\circ\sigma = \mathrm{Id}_{\mathbb H^n}$ and $\sigma(0,\ldots,0,1)=(0,\ldots,0,1)$, given by
\begin{equation*}
\sigma(w)=\mathcal{C} (-\mathcal F(w)) = \left( -\frac{w_1}{w_n},\ldots,-\frac{w_{n-1}}{w_n},\frac{1}{w_n}\right).
\end{equation*}
By a straight computation, we have
\begin{align*}
\sigma_*\left(\frac{\partial}{\partial w_k}\right) &= -w_n\frac{\partial}{\partial w_k}, \quad 1\leq k\leq n-1\\
\sigma_*\left(\frac{\partial}{\partial w_n}\right) &= -\sum_{j=1}^{n-1} w_nw_j\frac{\partial}{\partial w_j}-w_n^2\frac{\partial}{\partial w_n}
	= -w_n\left( D-w_n\frac{\partial}{\partial w_n}\right).
\end{align*}
Therefore
\begin{align*}
\widetilde{T}
	& :=\sigma_*(T) 
    = -2\sqrt{-1} w_n \left(D-w_n\frac{\partial}{\partial w_n}\right) \\
\widetilde{T}^{2,k}
	&:=\sigma_*(T^{2,k}) = 2 w_k \left(D-w_n\frac{\partial}{\partial w_n}\right) -w_n\frac{\partial}{\partial w_k} 
	 \\
\widetilde{T}^{3,k}
	&:=\sigma_*(T^{3,k}) 
	= -2\sqrt{-1} w_k \left( D-w_n\frac{\partial}{\partial w_n}\right) -\sqrt{-1} w_n\frac{\partial}{\partial w_k}.
\end{align*}
Since
\begin{equation*}
\left[D,D-w_n\frac{\partial}{\partial w_n}\right] = -\left[D,w_n\frac{\partial}{\partial w_n}\right] = 0,
\end{equation*}
we have
\begin{align*}
[D,\widetilde{T}] 
 &= -2\sqrt{-1} ( D w_n) \left(D-w_n\frac{\partial}{\partial w_n}\right) 
= -4\sqrt{-1} w_n \left(D-w_n\frac{\partial}{\partial w_n}\right)  =  2\widetilde{T},
	\\
[D,\widetilde{T}^{2,k}] 
	&= 2 ( D w_k) \left(D-w_n\frac{\partial}{\partial w_n}\right) -\left[ D, w_n\frac{\partial}{\partial w_k} \right]
	\\
	&= 2  w_k \left(D-w_n\frac{\partial}{\partial w_n}\right) - w_n\frac{\partial}{\partial w_k} = \widetilde{T}^{2,k},
	\\	
[D,\widetilde{T}^{3,k}] 
	&= -2\sqrt{-1} (D w_k)  \left(D-w_n\frac{\partial}{\partial w_n}\right) -\sqrt{-1} \left[ D,w_n\frac{\partial}{\partial w_k} \right]
	\\
	&= -2\sqrt{-1} w_k  \left(D-w_n\frac{\partial}{\partial w_n}\right) -\sqrt{-1} w_n\frac{\partial}{\partial w_k} = \widetilde{T}^{3,k}.
\end{align*}
Therefore
\begin{equation*}
\widetilde{T}\in\g_1 \quad \text{ and }\quad
\widetilde{T}^{2,k},\widetilde{T}^{3,k}\in\g_{1/2} \;.
\end{equation*}

\begin{proposition}\label{prop:1}
\begin{enumerate}
\item\label{lie algebra}
$\g_{-1} = \ip{T}$,
$\g_{-1/2} = \ip{T^{2,k},T^{3,k}}$,
$\g_0 = \ip{D,U^{ij},V^{ij},W^k}$,\\
$\g_{1/2}=\ip{ \widetilde T^{2,k}, \widetilde T^{3,k}}$ ,
$\g_{1}=\ip{\widetilde T}$,
\item \label{affine v.f.} For $V\in\aut(\mathbb H^n)$, $V\in\aff(\mathbb H^n)$ if and only if $(\RE V) \log\varphi_0\equiv c$.
\end{enumerate}
\end{proposition}
\begin{proof}
For \eqref{lie algebra} we remark that the vectors $T$, $D$, $T^{2,k}$, $T^{3,k}$, $U^{ij}$, $V^{ij}$, $W^k$, $\widetilde T$, $\widetilde T^{2,k}$ and $\widetilde T^{3,k}$ with $k=1,\ldots,n-1$, $1\leq i<j\leq n-1$ are linearly independent and the number of these vector fields is $n^2+2n$.
Since $\dim_{\mathbb R} \dim \Aut(\mathbb H^n)=n^2 + 2n$, we complete the proof.

    To prove \eqref{affine v.f.} define $\rho_0(w):= \RE w_n - |w'|^2$.
Since
\begin{align*}
\left(D-w_n\frac{\partial}{\partial w_n}\right) \rho_0 = w_n-|w'|^2-\frac{w_n}{2} = \frac{w_n}{2}-|w'|^2,
\end{align*}
we obtain
\begin{align*}
(\RE \widetilde{T})\rho_0 
	&= -2\sqrt{-1} w_n \left(\frac{w_n}{2}-|w'|^2\right) +2\sqrt{-1} \bar w_n\left( \frac{\overline w_n}{2}-|w'|^2\right)
	\\
	&= -\sqrt{-1} (w_n^2- \bar w_n^2) +2\sqrt{-1} (w_n-\bar w_n)|w'|^2
	= -2\sqrt{-1} (w_n-\bar w_n)\rho_0,
	\\
(\RE \widetilde{T}^{2,k})\rho_0 
	&= 2w_k\left(\frac{w_n}{2}-|w'|^2\right) +w_n\bar w_k + 2\bar w_k\left(\frac{\bar w_n}{2}-|w'|^2\right) + \bar w_n w_k
	\\
	&= (w_k+\bar w_k)(w_n+\bar w_n) -2(w_k+\bar w_k)|w'|^2 
	= 2(w_k+\bar w_k)\rho_0,
	\\
(\RE \widetilde{T}^{3,k})\rho_0 
	&= -2\sqrt{-1} w_k\left(\frac{w_n}{2}-|w'|^2\right) +\sqrt{-1} w_n\bar w_k 
	+ 2\sqrt{-1} \bar w_k\left(\frac{\bar w_n}{2}-|w'|^2\right) -\sqrt{-1} \bar w_n w_k
	\\
	&=-\sqrt{-1} (w_k-\bar w_k)(w_n+\bar w_n) +2\sqrt{-1}(w_k-\bar w_k)|w'|^2
	=-2\sqrt{-1} (w_k-\bar w_k)\rho_0.
\end{align*}
Since $\log\psi_0 = -(n+1)\log\rho_0$, we have 
\begin{equation*}
X\log\psi_0 = -(n+1)\frac{X\rho_0}{\rho_0}
\end{equation*}
for any real tangent vector field $X$ of $\mathbb H^n$,
thus
\begin{align*}
(\RE \widetilde{T})\log\psi_0 &= 2(n+1)\sqrt{-1} (w_n-\bar w_n)
=-4(n+1)\text{Im}\, w_n, \\
(\RE \widetilde{T}^{2,k})\log\psi_0 &= -2(n+1)(w_k+\bar w_k)=-4(n+1)\RE w_k, \\
(\RE \widetilde T^{3,k}
)\log\psi_0 &= -2\sqrt{-1} (n+1) (w_k-\bar w_k) = (n+1)\text{Im}\, w_k. 
\end{align*}
This completes the proof.
\end{proof}

\section{Proofs of Theorem~\ref{main2} and Theorem~\ref{main3}}\label{proof of main23}

\begin{proof}[Proof of Theorem~\ref{main2}]
Let $F:=G\circ \mathcal F\colon \mathbb H^n \to \Omega$ be a biholomorphism and $\omega_\Omega=(g_{i\overline j})$ be the K\"ahler form of the Bergman metric of $\Omega$.
 We may assume $(0,\ldots,0,1)\in \Omega$ and $dF(0)=I_n$. Let
$$
I_{1,2} := \left(\begin{array}{cccc}
    1 &&&  \\
     &\ddots&&\\
     &&1&\\
     &&&2
\end{array}\right).
$$
Let $\varphi:= K_\Omega$ be the Bergman kernel of $\Omega$. 
Since $\Omega$ is homogeneous, by the transformation formula of the Bergman kernel by automorphisms, we have 
\begin{equation}\nonumber
K_\Omega(z,z) = c\det (g_{i\overline j}(z))
\end{equation}
for some constant $c$.
This implies that 
$\omega_\Omega$ is K\"ahler--Einstein and
$\sqrt{-1}\partial\overline\partial\log\varphi=\omega_\Omega$.
Since $\Omega$ is affine homogeneous, we have $\|\partial\log\varphi\|^2_{\omega_\Omega}$ is a constant.
Let $\widetilde \varphi := \varphi\circ F$. Then $\sqrt{-1}\partial\overline\partial  \log\widetilde \varphi = \omega_{\mathbb H^n}$ and $\|\partial\log \widetilde\varphi\|^2_{\omega_{\mathbb H^n}}\equiv C$ for some constant $C$. By Corollary~\ref{thm_ball}, if we compose an isotropy of $\mathbb H^n$ at $(0,\ldots,0,1)$, we have $\widetilde\varphi = \psi_0$ and $C=n+1$.

If $V\in \aff(\Omega)$, then $\RE V\log \varphi \equiv c$ for some constant $c$. Thus Proposition~\ref{lie algebra} \eqref{affine v.f.} implies that $V\in F_* \aff(\mathbb H^n)$, i.e. $$
\aff(\Omega)\subset F_*\aff(\mathbb H^n).
$$
For $p=F(p_0)$ where $p_0=(0,\ldots,0,1)\in \mathbb H^n$, the affine homogeneity implies that 
$$
T_p\Omega = \{ V(p):V\in \aff(\Omega)\}.
$$
Since $\{D_{p_0}, T_{p_0}, T^{2,k}_{p_0}, T^{3,k}_{p_0}, k=1,\ldots, n-1\}$ is a basis of $T_{p_0}\mathbb H^n$ and $U_{p_0}$ vanishes for all unitary vector fields $U$ of $\mathbb H^n$, we obtain 
\begin{equation*}
\widehat{T}, \widehat{T}^{2,k}, \widehat{T}^{3,k}, \widehat{D}\in\aff(\Omega)
\end{equation*}
where $\widehat T:= F_*T$, $\widehat T^{2,k}:=F_* T^{2,k}$, $\widehat T^{3,k}:=F_* T^{3,k}$ and $\widehat D:= F_*D$.

Now write $\widehat T$ and $\widehat D$ as
\begin{equation}\nonumber
    \widehat T = \sum_{i=1}^n\left(\sum_{j=1}^n a_{ij}w_j + b_i \right)\frac{\partial}{\partial w_i},\quad
    \widehat D = \sum_{k=1}^n\left(\sum_{\ell=1}^n A_{k\ell}w_\ell + B_k\right) \frac{\partial}{\partial w_k}
\end{equation}
with constants $a_{ij}, b_i, A_{k\ell}, B_k\in \mathbb C$.
Denote $a=(a_{ij})$, $b=(b_i)$, $A=(A_{k\ell})$ and $B=(B_k)$ be the matrices corresponding to $a_{ij}$, $b_i$, $A_{k\ell}$ and $B_k$, respectively.
Let $JF$ denote the Jacobian matrix of $F$.

By $F_*T=\widehat T$ and $F_*D = \widehat D$, we have 
\begin{equation}\label{first eq}
    \begin{aligned}
        &2\sqrt{-1}\frac{\partial F_i}{\partial z_n} = \sum_{j=1}^na_{ij} F_j + b_i,\\
        &2z_n\frac{\partial F_i}{\partial z_n} + \sum_{\ell=1}^{n-1} z_\ell\frac{\partial F_i}{\partial z_\ell} = \sum_{j=1}^n A_{ij} F_j+B_i
    \end{aligned}
\end{equation}
for any $i=1,\ldots,n$.
By differentiating these equations with respect to $\frac{\partial}{\partial z_k}$ with $k=1,\ldots,n$ we obtain
\begin{equation}\label{second eq}
    \begin{aligned}
        &2\sqrt{-1}\frac{\partial }{\partial z_n} \frac{\partial F_i}{\partial z_k} = \sum_{j=1}^na_{ij} \frac{\partial F_j}{\partial z_k},\\
        &2z_n \frac{\partial}{\partial z_n}\frac{\partial F_i}{\partial z_k} +\frac{\partial F_i}{\partial z_k} +\sum_{\ell=1}^{n-1} z_\ell\frac{\partial^2 F_i}{\partial z_k\partial z_\ell} = \sum_{j=1}^n A_{ij}\frac{\partial F_j}{\partial z_k}, \quad \forall k=1,\ldots, n-1,\\
        &2z_n \frac{\partial}{\partial z_n}\frac{\partial F_i}{\partial z_n} +2\frac{\partial F_i}{\partial z_n} 
        +\sum_{\ell=1}^{n-1} z_\ell\frac{\partial^2 F_i}{\partial z_n\partial z_\ell} = \sum_{j=1}^n A_{ij}\frac{\partial F_j}{\partial z_n}.
    \end{aligned}
\end{equation}
At $(0,\ldots,0,z_n)$ as a matrix form we have 
\begin{equation}\nonumber
    \begin{aligned}
        &2\sqrt{-1}\frac{\partial }{\partial z_n} JF = aJF,\\
        &2z_n \frac{\partial}{\partial z_n} JF + I_{1,2} JF = A JF.
    \end{aligned}
\end{equation}
This implies that 
\begin{equation}\nonumber
    -\sqrt{-1}z_n aJF + I_{1,2} JF = AJF
\end{equation}
and hence 
\begin{equation}\nonumber
    -\sqrt{-1}z_na + I_{1,2} = A.
\end{equation}
By letting $z_n\to 0$, we obtain $A=I_{1,2}$ and hence $a=0$.
By $F(0,\ldots,0,1)=(0,\ldots,0,1)$, $dF(0)=I_n$ and \eqref{first eq}, 
\begin{equation}\nonumber
    \widehat T = T,\quad \widehat D=D.
\end{equation}
This implies that 
\begin{equation}\label{in}
    \frac{\partial F_i}{\partial z_n} = \delta_{in} 
\end{equation}
and by \eqref{second eq} we have
\begin{equation}\nonumber
    \sum_{\ell=1}^{n-1} z_\ell \frac{\partial^2 F_i}{\partial z_\ell \partial z_k} = 0.
\end{equation}
Now fix $z_n$ and express $\frac{\partial F_i}{\partial z_k}$ by homogeneous polynomials $P_j$ of degree $j$ in $z_1,\ldots, z_{n-1}$ as
\begin{equation}
    \frac{\partial F_i}{\partial z_k}
    = \sum_{j=0}^\infty P_j(z_1,\dots z_{n-1}).
\end{equation}
Then, since 
\begin{equation}
    0=\sum_{\ell=1}^{n-1} z_\ell \frac{\partial^2 F_i}{\partial z_\ell \partial z_k}=\sum_{j=1}^\infty jP_j(z_1,\ldots,z_{n-1}),
\end{equation}
we obtain $P_j\equiv 0$ for any $j=1,2,\ldots$. This implies 
$$\frac{\partial F_i}{\partial z_k}=\eta(z_n)$$ for some function $\eta$. 
Since by \eqref{in} 
\begin{equation}
0=\frac{\partial^2 F_i}{\partial z_k\partial z_n} = \frac{\partial \eta}{\partial z_n},
\end{equation}
$\eta$ is constant. 
Hence $\frac{\partial F_i}{\partial z_k}$ is a constant $\delta_{ik}$. With \eqref{in}, we obtain $F=id$. 
As a result $G=\mathcal C$ and $\Omega=\mathbb H^n$.

\end{proof}

\begin{proof}[Proof of Theorem~\ref{main3}]
Let $G\colon \mathbb B^n\to D$ be a biholomorphism. It is worth to emphasize that since $\omega_D$ is the pullback of the Bergman metric $\omega_{\mathbb B^n}$ on $\mathbb B^n$, $G$ is an isometry with respect to $\omega_{\mathbb B^n}$ and $\omega_D$. 
As a result we obtain
\begin{equation}\nonumber
    \sum_{\alpha,\beta=1}^n g_{\alpha\bar\beta}\circ G \frac{\partial G_\ell}{\partial z_\alpha}\frac{\partial\overline G_m}{\partial \overline z_\beta}
    =\frac{\partial^2}{\partial z_\ell \partial \overline z_m}\log K_{\mathbb B^n}(z,z)
\end{equation}
where $K_{\mathbb B^n}$ denotes the Bergman kernel of $\mathbb B^n$. By taking the determinant in both sides of the equation, 
one has 
\begin{equation}\label{isometry}
    \det g_{\alpha\overline\beta}\circ G |\det dG|^2 = \det (g_{\alpha\bar\beta}^{\mathbb B^n}) = K_{\mathbb B^n}(z,z) = \frac{C_1}{(1-|z|^2)^{n+1}}
\end{equation}
for some constant $C_1$.
Since $\omega_D$ is the K\"ahler--Einstein metric, we have $\partial\bar\partial \log\det (g_{\alpha\bar\beta})
=\omega_D$.
Without loss of generality we may assume that $(0,\ldots, 0,1)\in D$ and $G(0) = (0,\ldots, 0,1)$.

Since $
\| \partial\log\det (g_{\alpha\bar\beta})\|_{\omega_D}^2\equiv C_2$ for some constant $C_2$, one has 
\begin{equation}\nonumber
\| \partial\left(\log\det (g_{\alpha\bar\beta})\circ G\right)\|_{\omega_{\mathbb B^n}}^2\equiv C_2
\quad\text{ and }\quad 
\sqrt{-1}\partial\bar\partial \left(\log\det (g_{\alpha\bar\beta})\circ G\right) = \omega_{\mathbb B^n}.
\end{equation}
Define $\varphi:=  \det(g_{\alpha\bar\beta})\circ G$. 
By Theorem~\ref{uniqueness of constant} we have $C_2=n+1$ and by Corollary~\ref{thm_ball}, we have 
\begin{equation}\nonumber
    \det(g_{\alpha\bar\beta})\circ G
    = C_3\frac{|1-z_n|^{2(n+1)}}{(1-|z|^2)^{n+1}}
\end{equation}
for some constant $C_3\in \mathbb R$ by composing an isotropy of $\mathbb B^n$ at $0$.
Hence, by \eqref{isometry} we obtain
\begin{equation}\nonumber
    |\det dG|^2 = C_4|1-z_n|^{-2(n+1)}
\end{equation}
for some constant $C_4$ and therefore we have 
\begin{equation}\label{det dG}
    \det dG = C_5 (1-z_n)^{-(n+1)}
\end{equation}
for some constant $C_5$.
By the affine transformation on $\mathbb C^n$, we may assume that 
\begin{equation}\label{dG(0)}
dG(0) = \left(\begin{array}{ccccc}
1&0&\cdots&0\\
&\ddots&&\\
0&\cdots&1&0\\
0&\cdots&0&2\\
\end{array}\right)
\end{equation}
and hence $C_5 = 2$.

Now suppose that $G$ is a M\"obius transformation, i.e.
\begin{equation}\nonumber
    G(z) = \left(
    \frac{\sum a_{1,j}z_j + a_{1,n+1}}{\sum a_{n+1, j}z_j + a_{n+1, n+1}}, \ldots,
    \frac{\sum a_{n,j}z_j + a_{n,n+1}}{\sum a_{n+1, j}z_j + a_{n+1, n+1}}
    \right).
\end{equation}
Since $a_{n+1,n+1}\neq 0$, assume that $a_{n+1,n+1}=1$.
Since $G(0) = (0,\ldots, 0,1)$, we have $a_{j,n+1}=0$ for any $j=1,\ldots,n$ and $a_{n,n+1}=1$.
By \eqref{dG(0)} 
\begin{equation}\nonumber
    a_{ki}=\delta_{ki} 
\end{equation}
for all $k=1,\ldots, n-1$, $i=1,\ldots,n$.
As a result, we have 
\begin{equation}\nonumber
    G(z) = \frac{\left(
    z_1, \ldots,
    z_{n-1}, 
    \sum a_{n,j}z_j + a_{n,n+1}
    \right)}{\sum a_{n+1, j}z_j + a_{n+1, n+1}}.
\end{equation}
Since for $\sigma(z) := \sum a_{n+1, j}z_j + a_{n+1, n+1}$
\begin{equation}\label{dG(z)}
    \begin{aligned}
    dG(z) &= \frac{1}{\sigma(z)}\left(\begin{array}{ccccc}
    1&0&\cdots&0&a_{n,1}\\
    0&1&\cdots&0&a_{n,2}\\
    &&&&\\
    0&\cdots&0&1&a_{n,n-1}\\
    0&\cdots&&0&a_{n,n}
    \end{array}\right)\\
    &-\frac{1}{\sigma(z)^2}\left(\begin{array}{ccccc}
    a_{n+1,1}z_1&a_{n+1,1}z_2&\cdots&a_{n+1,1}z_{n-1}& a_{n+1,1}(\sum a_{n,j}z_j + a_{n,n+1})\\
    a_{n+1,2}z_1&a_{n+1,2}z_2&\cdots&a_{n+1,2}z_{n-1}& a_{n+1,2}(\sum a_{n,j}z_j + a_{n,n+1})\\
    &&&&\\
    a_{n+1,n}z_1&a_{n+1,n}z_2&\cdots&a_{n+1,n}z_{n-1}& a_{n+1,n}(\sum a_{n,j}z_j + a_{n,n+1})\\
    \end{array}\right),
    \end{aligned}
\end{equation}
at $z=(z_1,0,\ldots, 0)$ we have
\begin{equation}\nonumber
    dG(z_1,0,\ldots, 0) = \frac{1}{\sigma^2}
    \left(\begin{array}{c|ccc|c}
    \sigma - a_{n+1,1}z_1&0&\ldots&0&a_{n,1}\sigma - a_{n+1,1}(a_{n,1}z_1 + z_{n,n+1})\\\hline
    -a_{n+1,2}z_1& \sigma&&&a_{n,2}\sigma - a_{n+1,2}(a_{n,1}z_1 + z_{n,n+1})\\
    &&\ddots&&\\
    -a_{n+1,n-1}z_1& &&\sigma&a_{n,n-1}\sigma - a_{n+1,n-1}(a_{n,1}z_1 + z_{n,n+1})\\\hline
    -a_{n+1,n}z_1&0 &\ldots&0&a_{n,n}\sigma - a_{n+1,n}(a_{n,1}z_1 + z_{n,n+1}),
    \end{array}
    \right)
\end{equation}
for $\sigma:=a_{n+1, 1}z_1 + a_{n+1, n+1}$ and by \eqref{det dG} we obtain
\begin{equation}\nonumber
\begin{aligned}
    2\sigma^{n+2} &= a_{n+1,n+1}\left(\sigma a_{n,n}-a_{n+1,n}(a_{n,1}z_1 + a_{n,n+1})\right) \\
    &\quad\quad\quad\quad\quad\quad
    +a_{n+1,n}z_1(a_{n,1}a_{n+1,n+1} - a_{n+1,1}a_{n,n+1}).
\end{aligned}
\end{equation}
This implies that $a_{n+1,1}=0$.
By the similar argument we have $a_{n+1,j}=0$ for any $j=1,\ldots,n-1$.
By \eqref{dG(z)} one has
\begin{equation}\nonumber
    dG(z) = \frac{1}{\sigma(z)^2}
    \left(\begin{array}{ccc|c}
    \sigma(z)&&&a_{n,1}\sigma(z)\\
    &\ddots&&\\
    &&\sigma(z)&a_{n,n-1}\sigma(z)\\\hline
    -a_{n+1,n}z_1&\cdots& -a_{n+1,n}z_{n-1} & a_{n,n}\sigma(z) - a_{n+1,n}(\sum a_{n,j}z_j + a_{n,n+1})
    \end{array}\right).
\end{equation}
By \eqref{dG(0)}, $a_{nj}=0$ for any $j=1,\ldots, n-1$.
At $z=(0,\ldots,0,z_n)$ we have 
\begin{equation}\label{0,z_n}
    \det dG(z) = \frac{a_{n,n} a_{n+1,n+1} - a_{n+1,n} a_{n,n+1}}{(a_{n+1,n}z_n + a_{n+1,n+1})^{n+1}}
    =2(1-z_n)^{-1-n}.
\end{equation}
Since $dG(0)_{nn} =2$, we have 
\begin{equation}\nonumber
    \frac{a_{n,n} a_{n+1,n+1} - a_{n+1,n} a_{n,n+1}}{ a_{n+1,n+1}^{2}}=2.
\end{equation}
Hence \eqref{0,z_n} we obtain 
\begin{equation}\nonumber
    a_{n+1,n+1}^{2}(1-z_n)^{n+1} 
    = (a_{n+1,n}z_n + a_{n+1,n+1})^{n+1}
\end{equation}
which implies $a_{n+1,n+1}^{n-1}=1$ and $a_{n+1,n+1} = -a_{n+1,n}$. By taking a rotation, we obtain $G = \mathcal C$.

\end{proof}

\end{document}